\documentclass[reqno,10pt]{amsart}
\usepackage{color}

\newcommand{\PP}{\mathbb{P}}
\newcommand{\RR}{\mathbb{R}}

\renewcommand{\a}{\alpha}
\renewcommand{\b}{\beta}

\renewcommand{\d}{\delta}
\newcommand{\e}{\varepsilon}

\newcommand{\cM}{{\mathcal{M}}}

\newcommand{\s}{\sigma}

\newcommand{\iy}{\infty}
\newcommand{\dint}{\displaystyle\int}
\newcommand{\dyle}{\displaystyle}

\newtheorem{theorem}{Theorem}[section]
\newtheorem{lemma}[theorem]{Lemma}
\newtheorem{proposition}[theorem]{Proposition}
\newtheorem{corollary}[theorem]{Corollary}
\newtheorem{definition}[theorem]{Definition}
\newtheorem{example}[theorem]{Example}
\newtheorem{remark}[theorem]{Remark}
\newtheorem{Remarks}[theorem]{Remarks}

\numberwithin{equation}{section}

\newcommand{\supp}[1]{\mbox{{supp\/}}#1}

\author[E. Colorado]{E. Colorado$^{(1)}$}
\address{Departamento de Matem\'aticas, Universidad Carlos III de Madrid,
Avenida de la Universidad 30, 28911 Legan\'es, Madrid, Spain}
\email{ecolorad@math.uc3m.es}
\thanks{$(1)$ Partially supported by Research Projects of
MICINN-Spain (Refs. MTM2009-10878, MTM2010-18128).}

\author[D. Pestana]{D. Pestana$^{(2)}$}
\address{Departamento de Matem\'aticas, Universidad Carlos III de Madrid,
Avenida de la Universidad 30, 28911 Legan\'es, Madrid, Spain}
\email{dompes@math.uc3m.es}
\thanks{$(2)$
Partially supported by Research Project of MICINN-Spain (Ref. MTM2009-07800)}
\author[J.\ M. Rodr{\'\i}guez]{J. M. Rodr{\'\i}guez$^{(2),(3)}$}
\address{Departamento de Matem\'aticas, Universidad Carlos III de Madrid,
Avenida de la Universidad 30, 28911 Legan\'es, Madrid, Spain}
\email{jomaro@math.uc3m.es}
\thanks{${(3)}$ Partially supported by Research Project of CONACYT-Mexico (Ref. CONACYT-UAG I0110/62/10).}

\author[E. Romera]{E. Romera$^{(4)}$}
\address{Departamento de Matem\'aticas, Universidad Carlos III de Madrid,
Avenida de la Universidad 30, 28911 Legan\'es, Madrid, Spain}
\email{eromera@math.uc3m.es}
\thanks{$(4)$ Partially supported by Research Project of MICINN-Spain (Ref. MTM2010/00005/001).}

\date{\today}

\begin{document}

\title[Muckenhoupt inequality with three measures]
{Muckenhoupt inequality with three measures and applications to Sobolev orthogonal polynomials}

\begin{abstract}

We generalize the classical Muckenhoupt inequality with two measures to
three under appropriate  conditions. As a consequence, we prove a simple
characterization of the boundedness of the multiplication operator
and thus of the boundedness of the zeros and the asymptotic behavior of the Sobolev orthogonal polynomials, for a large class of measures
which includes the most usual examples in the literature.

\

\end{abstract}

\maketitle{}
\section{Introduction}

The starting point of our work is the classical Muckenhoupt inequality, i.e., there exists a  constant $c>0$ such that
\begin{equation} \label{muc-inequality0}
\|f\|_{L^q(\mu)} \le c \, \|f' \|_{L^{p}(\nu)}
\end{equation}
for any regular enough function $f$,  where
$1< p \le q < \iy$, and $\mu, \nu$ are nonnegative $\s$-finite Borel measures on $(0,\iy)$.
We are interested in considering three measures instead of
two in this inequality.
Precisely, we look for conditions on the measures $\nu_1, \nu_2, \nu_3$
for which it is true
\begin{equation} \label{maininequality0}
\|f\|_{L^p(\nu_1)} \le c \left( \|f\|_{L^{p}(\nu_2)} + \|f' \|_{L^{p}(\nu_3)}  \right)
\end{equation}
for all regular enough functions $f$ (see the precise statement in \eqref{eq:main}).
The inequality \eqref{maininequality0} is obviously true if $\nu_1\le k \nu_2$ for some constant $k$.

\medskip

We remember the precise result about the Muckenhoupt inequality with the necessary and sufficient condition in
order to \eqref{muc-inequality0} be satisfied.
\begin{theorem}[Muckenhoupt \cite{Mu}] \label{t:00}
Assume $1< p \le q < \iy$, let $\mu, \nu$ be nonnegative $\s$-finite Borel measures on $(0,\iy)$.
Then there exists a constant $C$ such that
\begin{equation} \label{muc-inequality00}
\left\|  \dint_0^x f(t) \, dt \right\|_{L^q((0,\iy), \mu)}
\le
C \| f\|_{L^p((0,\iy), \nu)}
\end{equation}
holds for all measurable functions $f$ in $(0,\infty)$ iff
\begin{equation}\label{muc-condition}
B:= \sup\limits_{r>0} \mu\left([r,\iy)\right)^{1/q}
 \left[\int_0^r \left(\dfrac {d \nu}{dt} \right)^{-1/(p-1)} dt \right]^{(p-1)/p} < \iy.
\end{equation}
\end{theorem}

\medskip

\begin{remark}
In \eqref{muc-condition} we assume the usual convention $0\cdot \infty=0$.
Along the paper, every density and singular part of any measure is considered with respect to the Lebesgue measure.
Note that, in fact, Muckenhoupt inequality \eqref{muc-inequality00} must be satisfied for all measurable functions $f$ such that
$\left\|  \int_0^x f(t) \, dt \right\|_{L^q((0,\iy), \mu)}$ makes sense (although it can be infinite); we will follow this Muckenhoupt convention.
\end{remark}

The classical  Muckenhoupt inequality \eqref{muc-inequality00} (see \cite{Mu})
 appears in many contexts of mathematics, see
 for example \cite[p. 40]{M}, where we find an equivalent condition for the estimate that some measures
must hold, that is in connection with the condition for the classical $A_p$ weights, see for instance \cite{HKM}.
Note that \eqref{muc-inequality00} is also related with the classical Hardy inequality, which is also known as an
expression of the Heisenberg uncertainty principle, first formulated as a principle of quantum mechanics in 1927,
see \cite{heisenberg}. Later on it was studied by other authors with different perspectives, see for example the classical
paper by Fefferman, \cite{fefferman}.

\medskip

In harmonic analysis, estimates of different operators with respect to weights have been largely studied;
in the classical
book \cite{GRF} we find a general presentation of the theory. The estimates on $L^p$ with one weight are
known for operators like the Hardy-Littlewood maximal operator or the Hilbert transform, for which we need
the $A_p$ weights. One can also
find strong estimates with two weights where one is obtained from the other. But
although the $A_p$ condition is generalized for pairs of weights, even for the Hardy-Littlewood maximal operator it
is not enough to obtain the strong estimate on $L^p$ with two weights;
this is a very active problem now in harmonic analysis.

\medskip

The field of application of our new Muckenhoupt inequality will be
weighted Sobolev spaces, and, in particular, the multiplication
operator (MO) defined by  $M f(z)=z\,f(z)$.
 In  \cite{Ku} these spaces are studied in the context of
partial differential equations. Also in approximation theory they are of great interest.
We will focus on this last topic and its relationship with Sobolev orthogonal polynomials (SOP).

\medskip

SOP have been widely investigated in the last years.
In particular, in \cite{IKNS1,IKNS2}, the authors showed that
the expansions with SOP can avoid the Gibbs
phenomenon which appears with classical orthogonal series in $L^2$ (see also \cite{MF}).
In \cite{RARP1,RARP2,RARP3} it was developed a theory of general Sobolev spaces with respect to measures on
the real line, in order to apply it to the study of SOP.
See \cite{APRR} for the generalization of this theory to curves in the complex plane.

\medskip

Our interest in the MO arises from its relationship with SOP controlling their zeros.
In the theory of SOP
we don't have the usual three term recurrence relation for orthogonal polynomials in $L^2$
so, it is really difficult to find an explicit expression for the SOP of degree $n$.
Hence, one of the central problems in the study of these
polynomials is to determine its asymptotic behavior. In \cite{LP}
it was shown how to obtain the $n$th root asymptotic of SOP if the zeros of these polynomials are
contained in a compact set of the complex plane. Although the
uniform bound of the zeros of orthogonal polynomials holds for every
measure with compact support in the case without derivatives,
it is an open problem to bound the zeros of SOP with respect to the norm
\begin{equation} \label{eq:first}
\|f\|_{W^{N,p}(\mu_0,\mu_1,\dots,\mu_N)} :=
\left( \sum_{k=0}^N \left\|f^{(k)}\right\|^{p}_{L^p(\mu_k)}
 \right)^{1/p},
\end{equation}
where $\mu_0, \mu_1, \dots , \mu_N$ are Borel measures and $p=2$.
The boundedness of the zeros is a
consequence of the boundedness of the MO
in $\PP^{N,p}(\mu_0,\mu_1,\dots,\mu_N)$ (the completion of the linear space of polynomials $\PP$
with respect to the norm \eqref{eq:first}); in fact,
the zeros of the SOP are contained in the
disk $\{z:\, |z|\le 2\|M\|\}$ (see \cite[Theorem 2.1]{LP}).

\medskip

If $p\neq 2$, then we have an analogue of SOP, precisely, we say that $q_{n}(z) = z^{n} + a_{n-1}z^{n-1} + \cdots+
a_{1}z + a_{0}$ is an $n$th monic extremal polynomial with respect
to the norm in \eqref{eq:first} if
$$
\|q_{n}\|_{W^{N,p}(\mu_0,\mu_1,\dots,\mu_N)}=\inf\left\{ \|q\|_{W^{N,p}(\mu_0,\mu_1,\dots,\mu_N)}: q(z)= z^{n} + b_{n-1}z^{n-1} + \cdots+ b_{1}z + b_{0}, \quad b_{j}\in \RR\right\}\,.
$$
It is clear that there exists at least an $n$th monic extremal polynomial. Furthermore, it is unique
if $1 < p < \infty$. If $p = 2$, then the $n$th monic extremal polynomial
is precisely the $n$th monic SOP with respect to the inner product
corresponding to $W^{N,2}(\mu_0,\mu_1,\dots,\mu_N)$.
In \cite{LPP1} the authors prove also for $1 < p < \infty$ that
the boundedness of the MO allows us to obtain
the boundedness of the zeros and the asymptotic behavior of the
extremal polynomials.
It is possible to generalize these results also in the context of ``nondiagonal" Sobolev norms (see \cite{LPP1,PQRT2,PRT}).

\medskip

In \cite{APRR,RARP2,RARP3,R2,R4,RS},
there are some answers to the
question stated in \cite{LP} about appropriate  conditions for $M$ to be bounded:
the most general results on this topic appear in \cite[Theorem 4.1]{R2} and \cite[Theorem 8.1]{APRR};
they characterize in a simple way (in terms of equivalent norms
in Sobolev spaces) the boundedness of $M$
in $\PP^{N,p}(\mu_0,\mu_1,\dots,\mu_N)$.
The rest of the papers mention several conditions which guarantee the equivalence of norms
in Sobolev spaces, and consequently, the boundedness of $M$.
However, these works have two objections: on the one hand, they require that the measures lead us to obtain a well
defined Sobolev space
(note that $W^{1,p}(\mu_0, \mu_1)$ is not well defined if $(\mu_1)_s \neq 0 $, since when the distributional derivative is a locally integrable function,
it is defined up to sets with zero Lebesgue measure);
on the other hand, they obtain conditions which guarantee the boundedness of $M$ if it is defined in the Sobolev space
$W^{1,p}(\mu_0, \mu_1)$ instead of $\PP^{1,p}(\mu_0, \mu_1)$.
In this paper we avoid these two objections.

\

We recall now the two classical definitions of Sobolev space on a compact interval
$I \subset \RR$ (with respect to the Lebesgue measure):

$(1)$ The Sobolev space $W^{1,p}(I)$ is the set of functions $f \in L^{p}(I)$ whose distributional
derivative is also a function in $L^{p}(I)$.

$(2)$ The Sobolev space $\PP^{1,p}(I):=H^{1,p}(I)$ is the completion with respect to the Sobolev norm of
$W^{1,p}(I)$ of the linear space of polynomials $\PP$  (or $\mathcal{C}^{k}(I)$ with $k\in \mathbb{N}$,
$\mathcal{C}^{\infty}(I)$, H\"older spaces, etc.).

Note that by construction in $(2)$ the spaces $\mathbb{P}$, $\mathcal{C}^{k}(I)$, etc, are  dense
in $\PP^{1,p}(I)$.

\

In 1964 it was shown by Meyers and Serrin, see \cite{MS}, that $H=W$, i.e., the previous definitions of Sobolev
space (with respect to the Lebesgue measure) are equivalent (see also \cite{A} and the references therein). In 1984,
Kufner and Opic showed in \cite{KO} that the situation is not so simple when one considers weights instead of the  Lebesgue measure;
however, if the weights $w_j$ verify $w_j^{-1/(p-1)}\in L^{1}$, then they give the right definition following the
philosophy of definition $(1)$.

\

Following the work \cite{KO}, in \cite{RARP1,RARP2} it appears a definition of Sobolev space for a large class of measures instead of weights.
For general measures, it is not possible to define $W^{1,p}(I)$, but it is possible to define the Sobolev space as
the completion $\PP^{1,p}(I)$ of the linear space of polynomials $\PP$.
(Note that  it is always possible to define the completion of a normed space as the set of equivalence classes of Cauchy sequences, which generate a Banach space).
Although the following is a very simple definition of Sobolev space, we show with this
example the difficulties about it:

\

Let us consider $\|f\|_{W^{1,2}([0,1],\mu_0,\mu_1)}^2 := \int_0^1 |f|^2 + |f(0)|^2 + |f'(0)|^2$.
If we only work with polynomials or for example $\mathcal{C}^1$-functions, this is a well-defined norm;
however, it has no meaning for other general sets of functions.
In order to determine the completion $\PP^{1,2}([0,1],\mu_0,\mu_1)$ of the polynomials with the norm
$\|\cdot\|_{W^{1,2}([0,1],\mu_0,\mu_1)}$
 note that
any function  $f \in L^2((0,1))$ may be
approximated in this norm  by
functions $g \in C^1([0,1])$ with
the values of $g(0)$ and $g'(0)$ fixed beforehand.
Therefore, the completion of the polynomials with respect to this norm
is isomorphic to $L^2([0,1])
\times \RR^2$.
Observe that given a function $g$ in $\mathcal{C}^1([0,1])$, there are infinitely
many equivalence classes in $L^2([0,1]) \times \RR^2$ whose restrictions
to $L^2([0,1])$ coincide almost everywhere with $g$.
This Sobolev space is a very strange object and it shows some difficulties in our study,
because we do not require a ``good behavior" of $\mu_0$ and $\mu_1$.
However, this kind of Sobolev norms appears in the study of SOP,
and the results in this paper allow us to prove that the MO is bounded with respect to this norm.

\medskip

The case of one derivative $(N=1)$, is the most usual in applications and in
the theory of Sobolev spaces and SOP. In that case,
the operator $M$ is bounded in $\PP^{1,p}(\mu_0, \mu_1)$ if and only if
$$
\|f\|_{W^{1,p}(\mu_0,\mu_1)} \asymp \|f\|_{W^{1,p}(\mu_0 +\mu_1, \mu_1)}
$$
for all $f \in \PP$, where the simbol $A\asymp B$ means that there exist two positive constants, $k_1$ and $k_2$, such that
$k_1 A \le B \le k_2 A$.
This is equivalent to
\begin{equation} \label{maininequality00}
\|f\|_{L^p(\mu_1)} \le c \left( \|f\|_{L^{p}(\mu_0)} + \left\|f'\right\|_{L^{p}(\mu_1)}  \right)
\end{equation}
for every $f\in \mathbb{P}$ and some constant $c$.

\medskip

That is the main reason why we deal with three measures instead of two in inequality \eqref{maininequality0}.

\medskip

The paper contains four more sections. In section \ref{sec:prel-not} we establish some notation and preliminaries.
Section \ref{tercera} deals with the generalized Muckenhoupt inequality \eqref{eq:main};
Proposition \ref{p:first} provides a very simple sufficient condition.
Theorems \ref{t:iff}, \ref{t:iff2}, \ref{t:iff3} and \ref{t:niff}
give several different hypotheses for which this condition is also necessary.
We also prove in Theorem \ref{l:polMuckenhoupt} that for finite measures \eqref{eq:main} holds for every measurable function iff it
holds for every polynomial. A counter-example in which the Muckenhoupt inequality \eqref{eq:main} is not satisfied is shwon in section
 \ref{cuarta}. Finally, applying the theorems in Section \ref{tercera} we obtain several results in Section \ref{quinta} about the MO.
In particular, the sum of Theorem \ref{t:subinterval} and Corollary \ref{t:subinterval3} characterizes the
boundedness of the MO for a large class of measures which includes the most usual examples in the literature
of orthogonal polynomials (see Example \ref{exam1}).
Furthermore, Theorems \ref{t:subinterval2} and \ref{t:subinterval4} and Corollary \ref{t:subinterval5}
give sufficient conditions in order to obtain the boundedness of the MO for a wider class of measures
(see Example \ref{exam2}).

\

\section{Notation and preliminaries}\label{sec:prel-not}

{\bf Notation.}
Along the paper we just consider nonnegative $\s$-finite Borel measures on $\mathbb{R}$. Besides:
\begin{itemize}
\item We assume that $1<p<\infty$ in  the whole work, so we omit it to simplify.
\item Measures are denoted by  $\nu_j$ or $\mu_j$, and densities (with respect to the Lebesgue measure) by $w_j$.
\item $(\nu )_s$ or $(\mu)_s$  denote singular parts,
and $(\nu)_{ac}$ or $(\mu)_{ac}$ absolutely continuous parts, with respect to the Lebesgue measure.
\item $(\nu_1-\nu_2)_+$ denotes the positive part of $\nu_1-\nu_2$.
\item Given a measurable set $A\subset \mathbb{R}$, we define the space of measurable functions
$$
\cM (A)=\{ f:A\to \mathbb{R}\, | \,  f\mbox{ is measurable on } A\}.
$$
\item For a measurable set $A\subset \mathbb{R}$, we denote by $I_A$ the characteristic function of $A$.
\item If $b\in \mathbb{R}$, $\delta_b$  denotes the Dirac delta measure concentrated at $\{b\}$.
\item For two finite measures $\mu_0, \mu_1$ on $[a,b]$,
we denote by $\PP^{1,p}([a,b],(\mu_0, \mu_1))$ or simply $\PP^{1,p}(\mu_0, \mu_1)$
the completion of the linear space of polynomials $\PP$ with respect to the Sobolev norm
$\|\cdot\|_{W^{1,p}([a,b],(\mu_0, \mu_1))}$.
\end{itemize}

\begin{remark}\label{rem:positive-part}
In general, $(\nu_1-\nu_2)_+$ makes sense if $\nu_1$ and $\nu_2$ are finite measures; however, it is possible to define
$(\nu_1-\nu_2)_+$ for $\s$-finite measures as follows.
Let us consider two increasing sequences of measurable sets $\{X_n^j\}$ with $\nu_j\left(X_n^j\right) < \infty$, $X^j:=\cup_n X_n^j$ and $\nu_j\left(\RR \setminus X^j\right) =0$
($j=1,2$).
Therefore, $\{X_n^1 \cap X_n^2\}$ is an increasing sequence of measurable sets with $\nu_j\left(X_n^1 \cap X_n^2\right) < \infty$ and $X^1\cap X^2=\cup_n \left(X_n^1 \cap X_n^2\right)$,
and it is possible to define the total variation $|\nu_1-\nu_2|$ of $\nu_1-\nu_2$, and its positive and negative parts $(\nu_1-\nu_2)_+$, $(\nu_1-\nu_2)_-$
as nonnegative $\s$-finite measures on $X^1 \cap X^2$.
Also, it is possible to define $|\nu_1-\nu_2|:=(\nu_1-\nu_2)_+:=\nu_1$ on $X^1 \setminus X^2$,
$|\nu_1-\nu_2|:=(\nu_1-\nu_2)_-:=\nu_2$ on $X^2 \setminus X^1$, and
$|\nu_1-\nu_2|:=0$ on $\RR \setminus \left(X_n^1 \cup X_n^2\right)$.
Hence, $|\nu_1-\nu_2|$, $(\nu_1-\nu_2)_+$ and $(\nu_1-\nu_2)_-$ are nonnegative $\s$-finite measures on $\RR$,
although $(\nu_1-\nu_2)(E)$ is not defined when $\nu_1(E)=\nu_2(E)=\infty$.
\end{remark}

\medskip

We want to generalize the Muckenhoupt condition \eqref{muc-condition} to the case
of three measures (with exponents $p=q$) and to fix our interest in an interval $[a,b]$
instead of $(0,\iy)$. In order to do this, we rewrite $B$ as follows.

\begin{definition}
Let $\nu_1, \nu_2$ be measures and $w_2= \dfrac {d\nu_2}{dx}$. We define:
$$
\begin{aligned}
\Lambda_{p,a}(\nu_1,\nu_2) & := \Lambda_{p,[a,b],a}(\nu_1,\nu_2) := \sup\limits_{r\in(a,b)} \nu_1\left([a,r]\right)
\left(\int_r^b w_2(t)^{-1/(p-1)} dt \right)^{p-1} ,
\\
\Lambda_{p,b}(\nu_1,\nu_2) & := \Lambda_{p,[a,b],b}(\nu_1,\nu_2) := \sup\limits_{r\in(a,b)} \nu_1\left([r,b]\right)
\left(\int_a^r w_2(t)^{-1/(p-1)} dt \right)^{p-1} ,
\\
\Lambda'_{p,b}(\nu_1,\nu_2) & := \Lambda'_{p,[a,b],b}(\nu_1,\nu_2) := \sup\limits_{r\in(a,b)} \nu_1\left([r,b)\right)
\left(\int_a^r w_2(t)^{-1/(p-1)} dt \right)^{p-1} .
\end{aligned}
$$
\end{definition}

\medskip

Note that Theorem \ref{t:00} in our setting, i.e., with  $(0,\infty)$ replaced by $(a,b)$ and $1<q=p<\infty$, can be read as follows.

\begin{theorem} \label{t:Muckenhoupt0}
Let $\nu_1, \nu_2$ be measures on $(a,b)$.
There exists a constant $C>0$ such that
\begin{equation}\label{muc-inequality}
\left\|  \dint_a^x f(t) \, dt \right\|_{L^p([a,b], \nu_1)}
\le C \|  f \|_{L^p([a,b], \nu_2)}
\end{equation}
holds for all $f \in \cM ([a,b])$
iff
$$
\Lambda'_{p,b}(\nu_1,\nu_2) <\iy\,.
$$
\end{theorem}

\medskip

In order to apply our results to SOP, we need to deal with measures on the compact interval $[a,b]$.
Hence, we need the following version of Theorem \ref{t:Muckenhoupt0} for compact intervals.

\medskip

\begin{theorem} \label{t:Muckenhoupt}
Let $\nu_1, \nu_2$ be measures on $[a,b]$.
There exists a constant $C>0$ such that
$$
\left\|  \dint_a^x f(t) \, dt \right\|_{L^p([a,b], \nu_1)}
\le C \left\|  f \right\|_{L^p([a,b], \nu_2)}
$$
holds for all $f \in \cM ([a,b])$ iff
$$
\Lambda_{p,b}(\nu_1,\nu_2) <\iy\,.
$$
\end{theorem}

\medskip

\begin{remark}
A similar result holds replacing $b$ by $a$.
Along the paper, most of the results will be stated just for one endpoint of the interval,
but they also hold for the other one by symmetry.
\end{remark}

\medskip

\begin{proof}[Proof of Theorem \ref{t:Muckenhoupt}]
Fix any measurable subset $S$ of $[a,b]$ with
zero Lebesgue measure and such that $(\nu_2)_s|_S=(\nu_2)_s$.

First of all, note that the singular part of $\nu_2$ does not play any role in $\Lambda_{p,b}(\nu_1,\nu_2)$.
For $f\in\mathcal{M}([a,b])$ we define the function $f_0:=f I_{[a,b]\setminus S}$. Then we have
$$
\left\|  \dint_a^x f_0(t) \, dt \right\|_{L^p([a,b], \nu_1)}
= \left\|  \dint_a^x f(t) \, dt \right\|_{L^p([a,b], \nu_1)}\,,
\qquad
\left\|  f_0 \right\|_{L^p([a,b], \nu_2)}= \left\|  f \right\|_{L^p([a,b], (\nu_2)_{ac})}\,,
$$
and we conclude that the singular part of $\nu_2$ does not play any role in \eqref{muc-inequality}.
Furthermore, if $F(x):=\int_a^x f(t)\, dt$, it verifies $F(a)=0$.
Hence, if $\nu_1$ is a measure on $[a,b)$ and $\nu_2$ is a measure on $[a,b]$, then Theorem \ref{t:Muckenhoupt0}
proves that
\eqref{muc-inequality} holds iff $\Lambda'_{p,b}(\nu_1,\nu_2) <\iy$.

Now, let's observe that
\begin{align}\label{comp-Lambdas}
\max\Big\{\Lambda'_{p,b}(\nu_1,\nu_2), & \ \nu_1(\{b\}) \Big(\int_a^b w_2(t)^{-1/(p-1)} dt \Big)^{p-1} \Big\}
\\
& \le \Lambda_{p,b}(\nu_1,\nu_2) \le
\Lambda'_{p,b}(\nu_1,\nu_2) + \nu_1(\{b\}) \Big(\int_a^b w_2(t)^{-1/(p-1)} dt \Big)^{p-1} \,. \notag
\end{align}
Therefore, if $\nu_1(\{b\})=0$ then $\Lambda_{p,b}(\nu_1,\nu_2)=\Lambda'_{p,b}(\nu_1,\nu_2)$ and we are done. So, let us suppose that
$\nu_1(\{b\})>0$.

\smallskip

If $\Lambda_{p,b}(\nu_1,\nu_2)<\infty$ then also
$\Lambda_{p,b}(\nu_1-\nu_1(\{b\})\d_b,\nu_2)=\Lambda'_{p,b}(\nu_1,\nu_2)<\infty$ and by Theorem \ref{t:Muckenhoupt0}
there exists a constant $C$ such that
$$
\left\| \int_a^x f(t)\, dt \right\|_{L^p([a,b],\nu_1-\nu_1(\{b\})\delta_{b})} \le C\, \left\|f\right\|_{L^p([a,b],\nu_2)} \,,
\qquad
\forall f \in \cM ([a,b])\,.
$$
Hence, in order to obtain (\ref{muc-inequality})
it suffices to prove that there exists a constant $C_1$ such that

\noindent $\|\int_a^x f\|_{L^p(\{b\},\,\delta_b)}\le C_1\,\|f\|_{L^p([a,b],\nu_2)}$ for all $f \in \cM ([a,b])$ or, equivalently, that
$$
\left|\int_a^b f(t)\, dt \right| \le C_2\,\left\|f\right\|_{L^p([a,b],\nu_2)}\, ,\qquad
\forall f \in \cM ([a,b])\,.
$$
By  H\"older inequality,
$$
\left|\int_a^b f(t)\, dt \right| = \left|\int_a^b f(t)w_2(t)^{1/p}w_2(t)^{-1/p}\, dt \right| \le
\left\|f\right\|_{L^p([a,b],\nu_2)} \left\|w_2^{-1/(p-1)}\right\|^{(p-1)/p}_{L^{1}([a,b])} \,, \quad \forall f \in \cM ([a,b])\,.
$$
Since $\nu_1(\{b\})>0$ and $\Lambda_{p,b}(\nu_1,\nu_2)<\infty$ imply that $\left\|w_2^{-1/(p-1)}\right\|_{L^{1}([a,b])}<\infty$
(see (\ref{comp-Lambdas})), we deduce (\ref{muc-inequality}).

Assume now that  (\ref{muc-inequality}) holds then, in particular,
$$
\left\| \int_a^x f(t)\, dt \right\|_{L^p([a,b],\nu_1-\nu_1(\{b\})\delta_{b})} \le C\, \left\|f\right\|_{L^p([a,b],\nu_2)},\qquad
\forall f \in \cM ([a,b])\,.
$$
Using again Theorem \ref{t:Muckenhoupt0}, we deduce that $\Lambda'_{p,b}(\nu_1,\nu_2)\allowbreak<\infty$.
Also, from (\ref{muc-inequality}), we have that

\noindent
$\nu_1(\{b\})^{1/p} \left| \int_a^b f(t)\, dt \right| \le C\, \left\|f\right\|_{L^p([a,b],\nu_2)}$ and therefore
$$
\left| \int_a^b f(t)\, dt \right| \le C'\, \left\|f\right\|_{L^p([a,b],\nu_2)},\qquad
\forall f \in \cM ([a,b])\,.
$$
 In particular, if we define $f_\e= \max\{w_2,\e\}^{-1/(p-1)} I_{[a,b]\setminus S}$, for $\e>0$, we obtain
$$
\int_a^b f_\e(t)\, dt \le C' \left(\int_a^b \max\{w_2(t),\e\}^{-p/(p-1)} \, w_2(t)\, dt\right)^{1/p} \le C'
\left(\int_a^b \max\{w_2(t),\e\}^{-1/(p-1)} dt\right)^{1/p} <\infty\,.
$$
Therefore,
$$
\left(\int_a^b \max\{w_2(t),\e\}^{-1/(p-1)} dt\right)^{(p-1)/p} \le C' .
$$
By the monotone convergence Theorem,  $\int_a^b w_2(t)^{-1/(p-1)} dt <\infty$ and, by (\ref{comp-Lambdas}),  $\Lambda_{p,b}(\nu_1,\nu_2)<\infty$.
\end{proof}

\medskip

\section{Muckenhoupt inequality with three measures}\label{tercera}

Let us start with a first approach to our problem, which gives a sufficient condition for \eqref{maininequality0}.
We will prove in Sections 3 and 4 that, in many situations, this condition is also necessary.

\medskip

\begin{proposition}\label{p:first}
Let  $\nu_1, \nu_2,\nu_3$ be measures on $[a,b]$.
Assume that
\begin{equation}\label{eq:mc}
\Lambda_{p,b} \left( (\nu_1-k\nu_2)_+, \nu_3\right) < \iy \,,
\end{equation}
for some constant $k\ge 0$.
Then there exists a constant $C$ such that
\begin{equation} \label{eq:main}
\left\|\dint_a^x f(t) \, dt \right\|_{L^p([a,b], \nu_1)}
\le c \left( \left\|\dint_a^x f(t) \, dt \right\|_{L^{p}([a,b],\nu_2)} + \left\|f \right\|_{L^{p}([a,b], \nu_3)}  \right), \quad\forall\,
f\in\cM ([a,b]).
\end{equation}
\end{proposition}

\medskip

The hypothesis \eqref{eq:mc} includes the two known cases: when $\nu_1 \le
k \nu_2$, and when the Muckenhoupt condition is fulfilled for
$\nu_1$ and $\nu_3$, i.e., $ \Lambda_{p,b}(\nu_1, \nu_3) < \iy. $

\medskip

\begin{proof}
First of all, we have
$$
\nu_1(E)-\nu_2(E)=(\nu_1-\nu_2)(E)=(\nu_1-\nu_2)_+(E)-(\nu_1-\nu_2)_-(E) \le (\nu_1-\nu_2)_+(E)
$$
for every measurable set $E$ with $\nu_2(E) < \infty$.
Hence,
\begin{equation} \label{eq:g}
\int_a^b g(x) \, d\nu_1 - k \int_a^b g(x) \, d\nu_2
\le \int_a^b g(x) \, d(\nu_1-k\nu_2)_+
\end{equation}
for every simple $\nu_2$-integrable function $g$.
If $g$ is now a nonnegative $\nu_2$-integrable function and $0<\a<\b$, then $\nu_2\left(g^{-1}([\a,\b))\right) < \infty$,
and we deduce \eqref{eq:g}
by approximating $g$ by simple $\nu_2$-integrable functions increasing to $g$.

Without loss of generality we can assume that $\int_a^x f(t) \, dt \in L^{p}([a,b],\nu_2)$,
since otherwise the inequality holds.
Therefore, we deduce from \eqref{eq:g} for $g(x)=\left| \int_a^x f(t) \, dt\right|^p$,
$\Lambda_{p,b} \left( (\nu_1-k\nu_2)_+, \nu_3\right) < \iy$ and by Theorem \ref{t:Muckenhoupt}
that there exists a constant $c_0$ such that
$$
\left\|\dint_a^x f(t) \, dt \right\|_{L^p([a,b], \nu_1)}^p - k
 \left\|\dint_a^x f(t) \, dt \right\|_{L^{p}([a,b],\nu_2)}^p
\le \left\|\dint_a^x f(t) \, dt \right\|_{L^p([a,b], (\nu_1-k\nu_2)_+)}^p
\le c_0 \left\|f \right\|_{L^{p}([a,b], \nu_3)}^p ,
$$
for all $ f\in\cM ([a,b])$ with $\int_a^x f(t) \, dt \in L^{p}([a,b],\nu_2)$.
Then, taking $C:=\max\{k,c_0\}^{1/p}$ we conclude the proof.
\end{proof}

\medskip

In terms of Sobolev spaces this estimate can be read as:
$$
\|g - g(a)\|_{L^p(\nu_1)}
\le C \left(\|g - g(a)\|_{L^{p}(\nu_2)} + \|g'\|_{L^p( \nu_3)} \right) ,
$$
where $g\in W^{1,p}([a,b],(\nu_1+\nu_2, \nu_3))$, if this Sobolev space is well defined.

\medskip
In order to deal with the weights that one usually finds in applications, and to obtain a characterization of \eqref{eq:main} in terms of them, we will use the following definition.

\medskip

\begin{definition}
We say that an ordered pair of weights $(w_1,w_2)$ on $[a,b]$ is in the class
$\mathfrak{C}_a([a,b])$ if we have either
$$
\lim\limits_{x\to a^+} \dfrac{w_1(x)}{w_2(x)} =\iy
\qquad
or
\qquad
\limsup\limits_{x\to a^+} \dfrac{w_1(x)}{w_2(x)} <\iy \,.
$$
Similarly, $(w_1,w_2) \in \mathfrak{C}_b([a,b])$ if we have either
$$
\lim\limits_{x\to b^-} \dfrac{w_1(x)}{w_2(x)} =\iy
\qquad
or
\qquad
\limsup\limits_{x\to b^-} \dfrac{w_1(x)}{w_2(x)} <\iy \,.
$$
\end{definition}

\medskip

Note that if, for example, there exists the limit
$$
\lim\limits_{x\to a^+} \dfrac{w_1(x)}{w_2(x)} =L \in[0,\iy] \,,
$$
then $(w_1,w_2)\in\mathfrak{C}_a([a,b])$.

\medskip

\begin{remark}
The class $\mathfrak{C}_a([a,b])$ contains every pair of weights
obtained by the products of:
$$
|x-a|^{\a_1}\,, \quad \exp \left(-\beta|x-a|^{-\a_2}\right)\,,
\quad \left|\log \dfrac 1{|x-a|} \right|^{\a_3}\,, \quad
\left|\log \left|\log \left| \cdots \left|\log \dfrac 1{|x-a|}\right| \cdots \right| \right| \right|^{\a_4}\,,
$$
with $\a_j \in \RR$ and $\b \ge 0$.
\end{remark}

\medskip

\begin{theorem} \label{t:iff}
Let  $\nu_1, \nu_2, \nu_3$ be measures on $[a,b]$.
Assume that
$\nu_1$ is finite, $w_3^{-1/(p-1)} \in L^{1}([a,r])$ for every $r\in (a,b)$,
and $(w_1, w_2)\in\mathfrak{C}_b([a,b])$.
Suppose  also that there exist constants $\e_0,c_0>0$ verifying the following:
$(\nu_2)_s([b-\e_0,b])=0$ if $\dyle\lim_{x\to b^-} w_1(x)/w_2(x) = \iy$,
and $(\nu_1)_s \le c_0(\nu_2)_s$ on $[b-\e_0,b]$ if $\dyle\limsup_{x\to b^-} w_1(x)/w_2(x) < \iy$.
Then, there exists a constant $c$ such that
\begin{equation} \label{eq:iff}
\left\|\dint_a^x f(t) \, dt \right\|_{L^p([a,b], \nu_1)}
\le  c  \left(\left\|\dint_a^x f(t) \, dt \right\|_{L^{p}([a,b],\nu_2)}
+  \,\left\|f \right\|_{L^{p}([a,b], \nu_3)}\right),
\quad \forall\, f\in\cM ([a,b])
\end{equation}
iff
there exists a constant $k\ge 0$ such that
\begin{equation} \label{eq:iff2}
\Lambda_{p,b} \left( (\nu_1-k\nu_2)_+, \nu_3\right) < \iy \,.
\end{equation}
\end{theorem}

\medskip

The following result will be useful in the proof of Theorem \ref{t:iff}.

\begin{lemma} \label{General Lemma}
Let  $\nu_1, \nu_2$ be measures on $[a,b]$.
Assume that $\nu_1$ is finite and $w_2^{-1/(p-1)} \in L^{1}([a,r_0])$ for some $r_0\in (a,b)$.
Then,
$$
\Lambda_{p,[a,b],b}(\nu_1, \nu_2) <\iy
\quad
\Longleftrightarrow
\quad
\Lambda_{p,[r_0,b],b}(\nu_1, \nu_2) < \iy \,.
$$
\end{lemma}

\medskip

\begin{proof}[Proof of Lemma \ref{General Lemma}]
Let us define
$$
\Lambda_{p,[a,b],b,r_0}(\nu_1, \nu_2) :=
\sup\limits_{r\in (r_0,b)} \nu_1\left([r,b]\right) \left( \dint_{a}^r w_2^{-1/(p-1)} \right)^{p-1} .
$$

We are going to prove the lemma by showing the following equivalences:
\begin{equation}\label{eq:dosiff}
\Lambda_{p,[a,b],b}(\nu_1, \nu_2) <\iy
\quad
\Longleftrightarrow
\quad
\Lambda_{p,[a,b],b,r_0}(\nu_1, \nu_2) < \iy
\quad
\Longleftrightarrow
\quad
\Lambda_{p,[r_0,b],b}(\nu_1, \nu_2) < \iy\,.
\end{equation}

Note that, since $\nu_1$ is finite and $w_2^{-1/(p-1)} \in L^{1}\left([a,r_0]\right)$,
$$
\sup\limits_{r\in (a,r_0]} \nu_1([r,b]) \left( \dint_{a}^r w_2^{-1/(p-1)} \right)^{p-1}
\le \nu_1([a,b]) \left( \dint_{a}^{r_0} w_2^{-1/(p-1)} \right)^{p-1}=:J < \infty \,.
$$
Then, we deduce
$$
\Lambda_{p,[a,b],b,r_0}(\nu_1, \nu_2)
\le
\Lambda_{p,[a,b],b}(\nu_1, \nu_2)
\le
\max \left\{ \Lambda_{p,[a,b],b,r_0}(\nu_1, \nu_2) ,\, J \right\},
$$
and the first equivalence in \eqref{eq:dosiff} holds.

In order to prove the second one, let us define $K:= \int_{a}^{r_0} w_2^{-1/(p-1)}$,
$c_p:=1$ if $1< p \le 2$, $c_p:=2^{p-2}$ if $p > 2$,
and $L:= c_p\, K^{p-1} \, \nu_1([r_0,b])$.
We have
$$
\begin{aligned}
\sup\limits_{r\in (r_0,b)} \nu_1([r,b]) \left( \dint_{a}^r w_2^{-1/(p-1)} \right)^{p-1}
& = \sup\limits_{r\in (r_0,b)} \nu_1([r,b]) \left( K + \dint_{r_0}^r w_2^{-1/(p-1)} \right)^{p-1}
\\
& \le \sup\limits_{r\in (r_0,b)} \nu_1([r,b]) \, c_p
\left( K^{p-1} +
\left( \dint_{r_0}^r w_2^{-1/(p-1)} \right)^{p-1}  \right)
\\
& \le c_p\, K^{p-1} \, \nu_1([r_0,b])
+ c_p \sup\limits_{r\in (r_0,b)} \nu_1([r,b])
\left( \dint_{r_0}^r w_2^{-1/(p-1)} \right)^{p-1}
\\
& = L + c_p\, \Lambda_{p,[r_0,b],b}(\nu_1, \nu_2) \,.
\end{aligned}
$$
Then, we deduce
$$
\Lambda_{p,[r_0,b],b}(\nu_1, \nu_2)
\le
\Lambda_{p,[a,b],b,r_0}(\nu_1, \nu_2)
\le
L + c_p\, \Lambda_{p,[r_0,b],b}(\nu_1, \nu_2)\,,
$$
which proves the second equivalence in \eqref{eq:dosiff}.
\end{proof}

\medskip

\begin{proof}[Proof of Theorem \ref{t:iff}]
By Proposition \ref{p:first} it suffices to prove that \eqref{eq:iff} implies \eqref{eq:iff2}.
Therefore, let's assume that \eqref{eq:iff} holds.

$\bullet$ Suppose first that $\dyle\lim_{x\to b^-} w_1(x)/w_2(x) = \iy$.
Hence, we have $d\nu_2=w_2 \,dx$ on $[b-\e_0,b]$.
Let us choose $0<\e<\e_0$ with $w_1(x)/w_2(x) \ge (2c)^{p}$ for every $x\in [b-\e,b)$.
For any  $f\in\cM ([a,b])$ with $\supp f \subseteq [b-\e,b]$,
the following holds
$$
\begin{aligned}
c\,\left\|\dint_a^x f(t) \, dt \right\|_{L^p([a,b], \nu_2)}
& =  c  \left( \dint_{b-\e}^b \left|\dint_{b-\e}^x f(t) \, dt \right|^{p} w_2(x)\,dx \right)^{1/p}
\\
& \le  c  \left( \dint_{b-\e}^b \left|\dint_{b-\e}^x f(t) \, dt \right|^{p} (2c)^{-p} \,w_1(x)\,dx \right)^{1/p}
\\
& \le  \frac12 \,\left\|\dint_{b-\e}^x f(t) \, dt \right\|_{L^p([b-\e,b], \nu_1)} \,.
\end{aligned}
$$
Therefore, by \eqref{eq:iff}, we have for every $f\in\cM ([a,b])$ with $\supp f \subseteq [b-\e,b]$ that:
$$
\left\|\dint_{b-\e}^x f(t) \, dt \right\|_{L^p([b-\e,b], \nu_1)}
\le  2\,c \,\left\|f \right\|_{L^{p}([b-\e,b], \nu_3)} \,.
$$
Then, by Theorem \ref{t:Muckenhoupt} we obtain
$\Lambda_{p,[b-\e,b],b}(\nu_1, \nu_3) < \iy$.
Hence,  by Lemma \ref{General Lemma}:
$\Lambda_{p,[a,b],b}(\nu_1, \nu_3) < \iy$,
which is \eqref{eq:iff2} with $k=0$.

\medskip

$\bullet$ Assume now that $\dyle\limsup_{x\to b^-} w_1(x)/w_2(x) < \iy$.
Therefore, $(\nu_1)_s \le c_0\,(\nu_2)_s$ on $[b-\e_0,b]$
and there exist constants $k_0>0$ and $0<\e<\e_0$ with $w_1(x) \le k_0 w_2(x)$ for every $x\in [b-\e,b)$.
Hence, if we define $k:=\max\{c_0,k_0\}$, then $\nu_1 \le k\nu_2$ on $[b-\e,b]$
and $(\nu_1 - k\nu_2)_+=0$ on $[b-\e,b]$.
Thus,
$$
\begin{aligned}
\Lambda_{p,b}\left((\nu_1 - k\nu_2)_+,\nu_3\right) & = \sup\limits_{r\in(a,b)} (\nu_1 - k\nu_2)_+\left([r,b]\right)
\left(\int_a^r w_3(t)^{-1/(p-1)} dt \right)^{p-1}
\\
& = \sup\limits_{r\in(a,b-\e)} (\nu_1 - k\nu_2)_+\left([r,b]\right)
\left(\int_a^r w_3(t)^{-1/(p-1)} dt \right)^{p-1}
\\
& \le (\nu_1 - k\nu_2)_+\left([a,b]\right)
\left(\int_a^{b-\e} w_3(t)^{-1/(p-1)} dt \right)^{p-1}
\\
& \le \nu_1 \left([a,b]\right)
\left(\int_a^{b-\e} w_3(t)^{-1/(p-1)} dt \right)^{p-1}
< \infty \,.
\end{aligned}
$$
\end{proof}

\medskip

As a  consequence of Theorem \ref{t:iff} we have the following result.

\medskip

\begin{corollary}\label{c:iff}
Let  $\nu_1, \nu_2, \nu_3$ be measures on $[a,b]$.
Assume that
$\nu_1$ is finite, $w_3^{-1/(p-1)} \in L^{1}([a,r])$ for every $r\in (a,b)$,
and $(w_1, w_2)\in\mathfrak{C}_b([a,b])$.
Suppose  also that there exists $\e_0>0$ verifying
$(\nu_1)_s([b-\e_0,b])=(\nu_2)_s([b-\e_0,b])=0$.
Then, there exists a constant $c$ such that
$$
\left\|\dint_a^x f(t) \, dt \right\|_{L^p([a,b], \nu_1)}
\le  c  \left( \left\|\dint_a^x f(t) \, dt \right\|_{L^{p}([a,b],\nu_2)}
+ \left\|f \right\|_{L^{p}([a,b], \nu_3)}\right),
\quad \forall\, f\in\cM ([a,b])
$$
 iff
there exists a constant $k\ge 0$ such that
$$
\Lambda_{p,b} \left( (\nu_1-k\nu_2)_+, \nu_3\right) < \iy \,.
$$
\end{corollary}

\medskip

We also have a result similar to Theorem \ref{t:iff} if $w_3^{-1/(p-1)}\not\in L^1$.

\medskip

\begin{theorem} \label{t:iff2}
Let  $\nu_1, \nu_2, \nu_3$ be measures on $[a,b]$.
Assume that $w_3^{-1/(p-1)} \notin L^{1}(I)$ for every interval $I\subseteq [a,b]$,
and $(w_1, w_2)$ is a pair in the class $\mathfrak{C}_b([a,b])$.
Suppose also that there exist constants $\e_0,c_0>0$ verifying the following:
\begin{itemize}
\item $(\nu_2)_s([b-\e_0,b])=0$ if $\dyle\lim_{x\to b^-} \frac{w_1(x)}{w_2(x)} = \iy$ and
$(\nu_1)_s \le c_0(\nu_2)_s$ on $[b-\e_0,b]$ if
 $\dyle\limsup_{x\to b^-} \frac{w_1(x)}{w_2(x)}  < \iy$,
\item for each $\e>0$ there exists a constant $c_\e>0$ with
$\nu_1 \le c_\e\nu_2$ on $[a,b-\e]$.
\end{itemize}
Then, there exists a constant $c$ such that
\begin{equation} \label{eq:iff21}
\left\|\dint_a^x f(t) \, dt \right\|_{L^p([a,b], \nu_1)}
\le  c \left( \left\|\dint_a^x f(t) \, dt \right\|_{L^{p}([a,b],\nu_2)}
+ \left\|f \right\|_{L^{p}([a,b], \nu_3)}\right),
\quad \forall\, f\in\cM ([a,b])\,,
\end{equation}
iff there exists a constant $k\ge 0$ such that
\begin{equation}\label{eq:mc2}
\Lambda_{p,b} \left( (\nu_1-k\nu_2)_+, \nu_3\right) < \iy.
\end{equation}
\end{theorem}

\medskip

\begin{proof}

By Proposition \ref{p:first} it suffices to prove \eqref{eq:mc2} assuming that \eqref{eq:iff21} holds.

\smallskip

If we have $\dyle\lim_{x\to b^-} w_1(x)/w_2(x) = \iy$, then,
as in the proof of Theorem \ref{t:iff}, we can choose $0<\e<\e_0$ such that
for any  $f\in\cM ([a,b])$ with $\supp f \subseteq [b-\e,b]$,
$$
\left\|\dint_{b-\e}^x f(t) \, dt \right\|_{L^p([b-\e,b], \nu_1)}
\le  2\,c \,\left\|f \right\|_{L^{p}([b-\e,b], \nu_3)} \,.
$$
Then, by Theorem \ref{t:Muckenhoupt} we obtain
$\Lambda_{p,[b-\e,b],b}(\nu_1, \nu_3) < \iy$.
Since $w_3^{-1/(p-1)} \notin L^{1}(I)$ for every interval $I\subseteq [a,b]$,
$\nu_1=0$ on $(b-\e,b]$.
By hypothesis, there exists a constant $k\ge 0$ with
$\nu_1 \le k\nu_2$ on $[a,b-\e]$.
Hence, we conclude $\nu_1 \le k\nu_2$ on $[a,b]$,
$(\nu_1 - k\nu_2)_+=0$ on $[a,b]$, and $\Lambda_{p,b} \left( (\nu_1-k\nu_2)_+, \nu_3\right)=0$.

\smallskip

If  we suppose now that  $\dyle\limsup_{x\to b^-} w_1(x)/w_2(x) < \iy$, then
we have $(\nu_1)_s \le c_0(\nu_2)_s$ on $[b-\e_0,b]$
and there exist constants $k_0>0$ and $0<\e<\e_0$ with $w_1(x) \le k_0 w_2(x)$ for every $x\in [b-\e,b)$.
Thus, taking $k_1:=\max\{c_0,k_0\}$, we have $\nu_1 \le k_1\nu_2$ on $[b-\e,b]$.
Since $\nu_1 \le c_\e\nu_2$ on $[a,b-\e]$,
if we define $k:=\max\{c_\e,k_1\}$, then $\nu_1 \le k\nu_2$ on $[a,b]$
and $(\nu_1 - k\nu_2)_+=0$ on $[a,b]$, and finally, $\Lambda_{p,b} \left( (\nu_1-k\nu_2)_+, \nu_3\right)=0$.
\end{proof}

\medskip

For the case of weights comparable to monotone functions we show in the next theorem sufficient conditions to obtain
our estimate.

\medskip

\begin{theorem} \label{t:iff3}
Let  $\nu_1, \nu_2$ be measures on $[a,b]$.
Assume that we have either

\smallskip

$(1)$ $(\nu_1)_s\le k_0(\nu_2)_s$ for some constant $k_0$
and $w_1$ is comparable to a
non-increasing function on $[a,b]$,

\noindent or

$(2)$ $\nu_1$ is a finite measure, $w_1^{-1/(p-1)}\in L^{1}([a,a+\e])$ for some $\e>0$,
and $w_1$ is comparable to a
non-decreasing function on $[a,b]$.

\smallskip

Then, there exists a constant $c$ such that
$$
\left\|\dint_a^x f(t) \, dt \right\|_{L^p([a,b], \nu_1)}
\le  c \left( \left\|\dint_a^x f(t) \, dt \right\|_{L^{p}([a,b],\nu_2)}
+ \left\|f \right\|_{L^{p}([a,b], \nu_1)}\right),
 \quad \forall\, f\in\cM ([a,b])\,,
$$
 and
\begin{equation} \label{eq:iff32}
\Lambda_{p,b} \left( (\nu_1-k\nu_2)_+, \nu_1\right) < \iy
\end{equation}
for some constant $k\ge 0$.
\end{theorem}

\medskip

\begin{proof}
By Proposition \ref{p:first}, it suffices to prove \eqref{eq:iff32}.
Without loss of generality we can assume that $w_1$ is a
 monotone function on $[a,b]$.

$\bullet$ Assume first that $(\nu_1)_s\le k_0(\nu_2)_s$ for some constant $k_0$
and that $w_1$ is a 
non-increasing function on $[a,b]$.
Then $(\nu_1-k_0\nu_2)_+ \le (\nu_1)_{ac}$, and it suffices to prove that
$\Lambda_{p,b} \left( (\nu_1)_{ac}, \nu_1\right) < \iy$, since then
\eqref{eq:iff32} holds.
We have
$$
\begin{aligned}
\Lambda_{p,b} \left( (\nu_1)_{ac}, \nu_1 \right)
& = \sup\limits_{r\in(a,b)} (\nu_1)_{ac}\left([r,b]\right)
\left(\int_{a}^r w_1(t)^{-1/(p-1)} dt \right)^{p-1}
\\
& = \sup\limits_{r\in(a,b)} \left(\int_{r}^b w_1(t)\, dt \right)
\left(\int_{a}^r w_1(t)^{-1/(p-1)} dt \right)^{p-1}
\\
& \le \sup\limits_{r\in(a,b)} w_1(r) \,(b - r) \, w_1(r)^{-1} (r-a)^{p-1}
\\
& \le (b -a)^{p} < \iy \,.
\end{aligned}
$$

$\bullet$ Assume now that $\nu_1$ is a finite measure, $w_1^{-1/(p-1)}\in L^{1}([a,a+\e])$ for some $\e>0$,
and $w_1$ is a 
non-decreasing function on $[a,b]$.
In this case we have $w_1(x) \ge w_1(a+\e) > 0$ for every $x \in [a+\e,b]$
and we conclude that
$w_1^{-1/(p-1)}\in L^{1}([a,b])$.
Therefore,  for any $k\ge 0$,
$$
\begin{aligned}
\Lambda_{p,b} \left( (\nu_1-k\nu_2)_+, \nu_1 \right)
&
\le \Lambda_{p,b} \left( \nu_1, \nu_1 \right)
= \sup\limits_{r\in(a,b)} \nu_1\left([r,b]\right)
\left(\int_{a}^r w_1(t)^{-1/(p-1)} dt \right)^{p-1}
\\
& \le \nu_1\left([a,b]\right)
\left(\int_{a}^b w_1(t)^{-1/(p-1)} dt \right)^{p-1} < \iy \,.
\end{aligned}
$$
\end{proof}

\medskip

We finish this section with a result on polynomial approximation for the Muckenhoupt inequality with three measures.

\medskip

\begin{theorem} \label{l:polMuckenhoupt}
Let $\nu_1, \nu_2, \nu_3$ be finite measures on $[a,b]$
with $w_3^{-1/(p-1)} \in L^{1}([a,r])$ for every $r\in (a,b)$.
Then there exists a constant $c>0$ such that the following inequality
\begin{equation} \label{eql:polMuckenhoupt}
\left\|  \dint_a^x f(t) \, dt \right\|_{L^p([a,b], \nu_1)}
\le
c \, \left( \left\|  \dint_a^x f(t) \, dt \right\|_{L^p([a,b], \nu_2)}
+ \left\|  f \right\|_{L^p([a,b], \nu_3)} \right)
\end{equation}
holds $\forall\, f\in\cM ([a,b])$  iff
it holds for any polynomial.
\end{theorem}

\medskip

\begin{proof}
Let us assume that \eqref{eql:polMuckenhoupt} holds for every polynomial and
define $c_1:=\max\big\{\nu_1([a,b])^{1/p},\, \nu_2([a,b])^{1/p},$ $\nu_3([a,b])^{1/p}\big\}$.
Fix a  function $f\in\cM ([a,b])$ and $\e>0$;
without loss of generality we can assume that $f$ belongs to $L^p([a,b], \nu_3)$.

\smallskip

Let's assume first that $f\in L^1([a,b])$.
The classical proof of the density of the continuous functions in $L^p$
(using the density of the simple functions and Lusin's Theorem) gives, in fact, that there exists a function $g\in \mathcal{C}([a,b])$
with
$$
\| f-g \|_{L^p([a,b], \nu_3)} + \| f-g \|_{L^1([a,b])} < \e\,.
$$
Weierstrass' Theorem provides a polynomial $q$ with $\left\| g-q \right\|_{L^\infty([a,b])} < \e$.
Hence,
$$
\begin{aligned}
\| g-q \|_{L^p([a,b], \nu_3)}
+ \| g-q \|_{L^1([a,b])}
& < \e\, \nu_3([a,b])^{1/p} + \e\, (b-a)
\le (c_1+b-a)\, \e
\,,
\\
\| f-q \|_{L^p([a,b], \nu_3)}
+ \| f-q \|_{L^1([a,b])}
& < (c_1+b-a+1)\, \e
=: c_2\, \e
\,,
\\
\left\|  \dint_a^x f(t) \, dt - \dint_a^x q(t) \, dt \right\|_{L^p([a,b], \nu_j)}
& \le
\|  f -q \|_{L^1([a,b])} \nu_j([a,b])^{1/p}
< c_1\,c_2\, \e\,.
\end{aligned}
$$
Since $\e>0$ is arbitrary, \eqref{eql:polMuckenhoupt} holds if $f\in L^1([a,b])$.

\smallskip

Now, let's suppose that $f\notin L^1([a,b])$.
Since $w_3^{-1/(p-1)} \in L^{1}([a,x])$ for every $x\in (a,b)$, we have
by H\"older inequality,
$$
\int_a^x \left|f(t)\right|\, dt = \int_a^x \left|f(t)\right|w_3(t)^{1/p}w_3(t)^{-1/p}\, dt \le
\left\|f\right\|_{L^p([a,b],\nu_3)} \left\|w_3^{-1/(p-1)}\right\|^{(p-1)/p}_{L^{1}([a,x])} < \infty \,,
$$
and then $f\in L^1([a,x])$ for every $x \in (a,b)$.
If the function $\int_a^x f$ has infinitely many zeros in any neighborhood of $b$, let $\{b_n\}$ be an increasing sequence
with $\int_a^{b_n} f=0$ and $\dyle\lim_{n \to \infty} b_n = b$;
otherwise, let $\{b_n\}$ be any increasing sequence with $\dyle\lim_{n \to \infty} b_n = b$.
Consider the sequence of functions $f_n:=f I_{[a,b_n]}\in L^1([a,b])$; we have proved
$$
\left\|  \dint_a^x f_n(t) \, dt \right\|_{L^p([a,b], \nu_1)}
\le
c \, \left( \left\|  \dint_a^x f_n(t) \, dt \right\|_{L^p([a,b], \nu_2)}
+ \left\|  f_n \right\|_{L^p([a,b], \nu_3)} \right)
$$
for every $n$.
Since $|  f_n |$ and $\left|  \int_a^x f_n(t) \, dt \right|$ increase with $n$,
\eqref{eql:polMuckenhoupt} holds for $f$ by the monotone convergence Theorem.

\end{proof}


\medskip

\section{A negative condition}\label{cuarta}

We show in this section a class of measures which do not satisfy the generalization of Muckenhoupt inequality \eqref{eq:main}.

\begin{theorem} \label{t:niff}
Let  $\nu_1, \nu_2, \nu_3$ be measures on $[a,b]$.
Assume that there exists $b_0\in [a,b)$ such that
$\nu_2([b_0,b]) < \infty$,
and $w_3^{-1/(p-1)} \notin L^{1}([r,b])$ for every $r\in (b_0,b)$.
If $\nu_1(\{b\})>0$ and $\nu_2(\{b\})=0$, then there is no constant $c$
for which
\begin{equation}\label{eq:mc3}
\left\|\dint_a^x f(t) \, dt \right\|_{L^p([a,b], \nu_1)}
\le  c \left(  \left\|\dint_a^x f(t) \, dt \right\|_{L^{p}([a,b],\nu_2)}
+ \left\|f \right\|_{L^{p}([a,b], \nu_3)}\right),
\quad \forall\, f\in\cM ([a,b])\,,
\end{equation}
with
$$
\Lambda_{p,b} \left( (\nu_1-k\nu_2)_+, \nu_3\right) = \iy, \qquad \forall \, k\ge 0.
$$
\end{theorem}

\begin{proof}
By Proposition \ref{p:first}, it suffices to prove the first statement.

\smallskip

Assume first that
$w_3^{-1/(p-1)} \in L^{1}([b_0,r])$ for every $r\in (b_0,b)$.
Since $\nu_1(\{b\})>0$, it suffices to show that
there does not exist any constant $c$ verifying
\begin{equation}\label{eq:mc4}
\left|\dint_a^b f(t) \, dt \right|
\le  c \left(  \left\|\dint_a^x f(t) \, dt \right\|_{L^{p}([a,b],\nu_2)}
+ \left\|f \right\|_{L^{p}([a,b], \nu_3)}\right) ,
 \quad \forall\, f\in\cM ([a,b]).
\end{equation}
Arguing by  contradiction, let us suppose  that there exists $c>0$ satisfying \eqref{eq:mc4}.

\smallskip

Let $S$ be a set with zero Lebesgue measure and such that $(\nu_3)_s|_S=(\nu_3)_s$.
For every natural number $n$, we define $a_n:=\max\{b_0,b-1/n\}$.
Since $w_3^{-1/(p-1)} \in L^{1}([b_0,r])$ and $w_3^{-1/(p-1)} \notin L^{1}([r,b])$ for every $r\in (b_0,b)$,
there exists $b_n\in (a_n,b)$ with $\int_{a_n}^{b_n} w_3^{-1/(p-1)}=n$;
let us define $f_n:=w_3^{-1/(p-1)} I_{[a_n,b_n]\setminus S}$.
By \eqref{eq:mc4} applied to $f_n$
$$
n
\le  c  \, n \, \nu_2([a_n,b])^{1/p}
+ c \left(\int_{a_n}^{b_n} w_3^{-p/(p-1)}w_3 \right)^{1/p}
\le  c  \, n \, \nu_2([a_n,b])^{1/p}
+ c \, n^{1/p} \,.
$$
Since $\nu_2([b_0,b]) < \infty$, we deduce that $\dyle\lim_{n\to\infty} \nu_2([a_n,b])=\nu_2(\{b\})=0$.
Hence, there exists some $n_0\in \mathbb{N}$ such that $c \,\nu_2([a_n,b])^{1/p} \le 1/2$,  $\forall\, n\ge n_0$.
Therefore, $n\le  2  \,  c \, n^{1/p}$, $\forall\, n\ge n_0$, which is a contradiction since $1<p<\infty$.
Then the conclusion holds if
$w_3^{-1/(p-1)} \in L^{1}([b_0,r])$ for every $r\in (b_0,b)$.

\medskip

We deal now with the general case.
Since $w_3^{-1/(p-1)} \notin L^{1}([r,b])$ for every $r\in (b_0,b)$,
we can choose an increasing sequence $\{c_n\}\subset (b_0,b)$ with $\dyle\lim_{n\to\infty} c_n= b$
and $\int_{c_{n-1}}^{c_{n}} w_3^{-1/(p-1)} \in [1,\infty]$ for every $n$.
Let $\{n_j\}$ be the ordered set of indices $n$ with $\int_{c_{n-1}}^{c_{n}} w_3^{-1/(p-1)} = \infty$, if any.
For each $j$, let us choose $\e_j\ge 0$ verifying
$$
1 \le \int_{c_{n_j-1}}^{c_{n_j}} \max\left\{w_3, \, \e_j \right\}^{-1/(p-1)} < \infty \,.
$$
If $\{n_j\} = \emptyset$, then we define $\overline{w}_3:=w_3$.
Otherwise, we define
$$
\overline{w}_3:=
\left\{
\begin{aligned}
w_3\,, \quad & \text{ on } [c_{n-1},c_{n}] \text{ if } n \notin \{n_j\} \,,
\\
\max\left\{w_3, \, \e_j \right\}\,, \quad & \text{ on } (c_{n_j-1},c_{n_j}) \text{ if } n = n_j \,,
\end{aligned}
\right.
$$
and $\overline{\nu}_3$ by $d\overline{\nu}_3 := \overline{w}_3 \, dx + d(\nu_3)_s$.

Note that
$$
\int_{c_{n-1}}^{c_{n}} \overline{w}_3^{-1/(p-1)} < \infty
\quad \text{ $\forall n$, and moreover} \quad
\int_{b_0}^{b} \overline{w}_3^{-1/(p-1)}
\ge \sum_{n \notin \{n_j\}} \int_{c_{n-1}}^{c_{n}} w_3^{-1/(p-1)} + \mbox{card}\, \{n_j\}
\,.
$$
As a consequence, $\overline{w}_3^{-1/(p-1)} \in L^{1}([b_0,r])$ and $\overline{w}_3^{-1/(p-1)} \notin L^{1}([r,b])$ for any $r\in (b_0,b)$.
Therefore, we have proved that \eqref{eq:mc3} is not satisfied with $\overline{\nu}_3$ instead of $\nu_3$.
Since $\nu_3 \le \overline{\nu}_3$, the conclusion holds for $\nu_3$.
\end{proof}

\medskip

\section{Application to Sobolev orthogonal polynomials}\label{quinta}

We start with the introduction of the concept of regular points, which will be the basis of the results of this section.

\medskip

\begin{definition}
Let $\mu_1$ be a measure on $[a,b]$.
If $w_1^{-1/(p-1)}\in L^1([a+\e,b-\e])$ for every $0<\e<(b-a)/2$, then we define the interval
of regular points $Reg([a,b])$ as follows:

$(1)$ In the case $w_1^{-1/(p-1)}\in L^{1}([a,b])$, then
$Reg([a,b])=[a,b]$.

\medskip

Moreover, if there exists $\e>0$ such that:

\medskip

$(2)$  $w_1^{-1/(p-1)}\in L^{1}([a,a+\e])$ and $w_1^{-1/(p-1)}\notin L^{1}([b-\e,b])$, then
$Reg([a,b])=[a,b)$.

$(3)$ $w_1^{-1/(p-1)}\notin L^{1}([a,a+\e])$ and $w_1^{-1/(p-1)}\in L^{1}([b-\e,b])$, then
$Reg([a,b])=(a,b]$.

$(4)$ $w_1^{-1/(p-1)}\notin L^1([a,a+\e])$ and $w_1^{-1/(p-1)}\notin L^1([b-\e,b])$, then
$Reg([a,b])=(a,b)$.
\end{definition}

\medskip

The concept of $Reg([a,b])$ is natural, as the following results show.

\medskip

\begin{theorem} \label{t:subinterval}
Let  $\mu_0, \mu_1$ be finite measures on $[a,b]$.
Assume that $w_1^{-1/(p-1)}\in L^{1}([a+\e,b-\e])$ for every $0<\e<(b-a)/2$.

\smallskip

$(1)$ In the case $Reg([a,b])= [a,b]$, then the MO  is bounded in $\PP^{1,p}(\mu_0, \mu_1)$
iff $\mu_0\left( [a,b] \right)>0$.

\smallskip

$(2)$ If $Reg([a,b])=[a,b)$, we assume also that
$(w_1, w_0)\in\mathfrak{C}_b([a,b])$
and $(\mu_0)_s([b-\e,b])=(\mu_1)_s([b-\e,b])=0$ for some $\e>0$.
Then the MO is bounded in $\PP^{1,p}(\mu_0, \mu_1)$
iff $\mu_0\left( [a,b] \right)>0$ and
\begin{equation} \label{eq:a02}
\Lambda_{p,[a,b],b} \left( (\mu_1-k\mu_0)_+,\,\mu_1 \right)
< \infty\,,
\end{equation}
for some constant $k$.

\smallskip

$(3)$ For $Reg([a,b])= (a,b]$, we assume also that
$(w_1, w_0)\in\mathfrak{C}_a([a,b])$
with $(\mu_0)_s([a,a+\e])=(\mu_1)_s([a,a+\e])=0$ for some $\e>0$. Then the MO is bounded in $\PP^{1,p}(\mu_0, \mu_1)$
iff $\mu_0\left( [a,b] \right)>0$ and
\begin{equation} \label{eq:a03}
\Lambda_{p,[a,b],a} \left( (\mu_1-k\mu_0)_+,\,\mu_1 \right)
< \infty\,,
\end{equation}
for some constant $k$.

\smallskip

$(4)$ When $Reg([a,b])= (a,b)$, we assume also that
$(w_1, w_0)\in\mathfrak{C}_a([a,b]) \cap\mathfrak{C}_b([a,b])$,
and $(\mu_0)_s([a,a+\e])=(\mu_1)_s([a,a+\e])=(\mu_0)_s([b-\e,b])=(\mu_1)_s([b-\e,b])=0$ for some $\e>0$.
Let us fix $x_0\in (a,b)$.
Then the MO is bounded in $\PP^{1,p}(\mu_0, \mu_1)$
iff $\mu_0\left( [a,b] \right)>0$ and
\begin{equation} \label{eq:a04}
\Lambda_{p,[a,x_0],a} \left( (\mu_1-k\mu_0)_+,\,\mu_1 \right)< \infty\,,
\quad
\Lambda_{p,[x_0,b],b} \left( (\mu_1-k\mu_0)_+,\,\mu_1 \right)
< \infty\,,
\end{equation}
for some constant $k$.
\end{theorem}

\medskip

\begin{remark}
Note that the hypotheses in Theorem \ref{t:subinterval}
imply $\mu_0\left( [a,b] \setminus Reg([a,b]) \right)=\mu_1\left( [a,b] \setminus Reg([a,b]) \right)=0$, i.e.,
$\mu_0|_{ Reg([a,b]) }=\mu_0$ and $\mu_1|_{ Reg([a,b]) }=\mu_1$.

We need finite measures since it is necessary to have
$\|g\|_{W^{1,p}([a,b],(\mu_0, \mu_1))} < \infty$, $\forall g\in \mathbb{P}$.

We also want to point  out that in \cite{CD} the compactness of the supports of the measures is
a necessary condition in order to have the boundedness of the MO.
\end{remark}

\medskip

Let us first prove the following lemmas, which will be useful in the proof of Theorem \ref{t:subinterval}.

\

\begin{lemma} \label{l:bp}
Let  $\mu_0, \mu_1$ be measures on $[a,b]$.
Assume that $0<\mu_0([a,b])<\infty$ and $w_1^{-1/(p-1)}\in L^{1}([a,b])$.
Then:

\smallskip

$(1)$ There exists a positive constant $c_1$ such that  $\forall\, c \in \mathbb{R}$, and all $ f \in \cM ([a,b])$:
$$
\left\|c+ \int_a^x f(t)\, dt \right\|_{L^\infty([a,b])}\le
c_1 \, \left(\left\| c+ \int_a^x f(t)\, dt \right\|_{L^p([a,b],\mu_0)}+\left\|f\right\|_{L^p([a,b],\mu_1)} \right).
$$

$(2)$ If $\mu_1$ is finite, then there exists a positive constant $c_2$ such that $\forall\, c \in \mathbb{R}$,
and all $ f \in \cM ([a,b])$ we have:
$$
\left\| c+ \int_a^x f(t)\, dt \right\|_{L^p([a,b],\mu_1)}\le
c_2 \, \left(\left\| c+ \int_a^x f(t)\, dt \right\|_{L^p([a,b],\mu_0)}+\left\|f\right\|_{L^p([a,b],\mu_1)} \right).
$$
\end{lemma}

\medskip

\begin{proof}
We just need to prove $(1)$, since $(2)$ is a direct consequence of $(1)$.

Let us fix $x_0\in [a,b]$.
For any $x\in [a,b]$ and $f \in \cM ([a,b])$, using H\"older inequality
$$
\begin{aligned}
\left| c+ \int_a^{x_0} f(t)\, dt \right|
&
\le \left| c+ \int_a^{x} f(t)\, dt \right|+\int_a^b \left|f\right|w_1^{1/p}w_1^{-1/p}
\\
& \le \left| c+ \int_a^{x} f(t)\, dt \right| + \left\|f\right\|_{L^p([a,b],w_1)}\|w_1^{-1/(p-1)}\|_{L^{1}([a,b])}^{\frac{p-1}{p}}
\\
& = \left| c+ \int_a^{x} f(t)\, dt \right|+c_3 \left\|f\right\|_{L^p([a,b],w_1)}\,,
\\
\left| c+ \int_a^{x_0} f(t)\, dt \right|^p
& \le 2^{p-1} \left(\left| c+ \int_a^{x} f(t)\, dt \right|^p+c_3^p \left\|f\right\|_{L^p([a,b],\mu_1)}^p\right)\,.
\end{aligned}
$$

Since $0<\mu_0([a,b]) < \infty$,
we can integrate in the x-variable on $[a,b]$ with respect to $\mu_0$ in order to find a constant $c_4>0$ such
 that for every $c\in\mathbb{R},$
$$
\left| c+ \int_a^{x_0} f(t)\, dt \right|^p\le c_4\,\left(
\left\| c+ \int_a^x f(t)\, dt \right\|_{L^p([a,b],\mu_0)}^p +\left\|f\right\|_{L^p([a,b],\mu_1)}^p \right), \quad \forall f\in \cM([a,b])
$$
 and all $x_0\in [a,b]$. Therefore, we conclude that for all $c\in \mathbb{R}$
$$
\left\|c+ \int_a^x f(t)\, dt \right\|_{L^\infty([a,b])}\le
c_1 \, \left(\left\| c+ \int_a^x f(t)\, dt \right\|_{L^p([a,b],\mu_0)}+\left\|f\right\|_{L^p([a,b],\mu_1)} \right),
\quad \forall\, f \in \cM ([a,b]).
$$
\end{proof}

\medskip

\begin{lemma} \label{l:last}
Let  $\mu_0, \mu_1$ be finite measures on $[a,b]$.
Assume that $w_1^{-1/(p-1)}\in L^{1}([a+\e,b-\e])$ for every $0<\e<(b-a)/2$ and that $Reg([a,b])= (a,b)$.
Let us fix $x_0\in (a,b)$.
If $\mu_0\left( (a,b) \right)>0$ and
\begin{equation} \label{eq:a04bis}
\Lambda_{p,[a,x_0],a} \left( (\mu_1-k\mu_0)_+,\,\mu_1 \right)< \infty\,,
\quad
\Lambda_{p,[x_0,b],b} \left( (\mu_1-k\mu_0)_+,\,\mu_1 \right)
< \infty\,,
\end{equation}
for some constant $k$, then the MO is bounded in $\PP^{1,p}(\mu_0, \mu_1)$.
\end{lemma}

\medskip

\begin{proof}
Proposition \ref{p:first} proves that there exists a constant $c$ such that
$$
\begin{aligned}
\left\|  \dint_{x_0}^x f(t) \, dt \right\|_{L^p([x_0,b], \mu_1)}
& \le
c \left( \left\|  \dint_{x_0}^x f(t) \, dt \right\|_{L^p([x_0,b], \mu_0)} +
\left\|  f \right\|_{L^p([x_0,b], \mu_1)} \right),
\\
\left\|  \dint_{x_0}^x f(t) \, dt \right\|_{L^p([a,x_0], \mu_1)}
& \le
c \left( \left\|  \dint_{x_0}^x f(t) \, dt \right\|_{L^p([a,x_0], \mu_0)} +
\left\|  f \right\|_{L^p([a,x_0], \mu_1)} \right) ,
\end{aligned}
$$
for all $f\in\cM ([a,b])$.
Then
$$
\begin{aligned}
\left\| g - g(x_0) \right\|_{L^p([x_0,b], \mu_1)}
& \le
c \left( \left\| g - g(x_0) \right\|_{L^p([x_0,b], \mu_0)} +
\left\|  g' \right\|_{L^p([x_0,b], \mu_1)} \right) ,
\\
\left\| g - g(x_0) \right\|_{L^p([a,x_0], \mu_1)}
& \le
c \left( \left\| g - g(x_0) \right\|_{L^p([a,x_0], \mu_0)} +
\left\|  g' \right\|_{L^p([a,x_0], \mu_1)} \right) ,
\end{aligned}
$$
are satisfied for every  $g\in \mathbb{P}$.
Consequently,
\begin{equation} \label{eq:a6}
\begin{aligned}
 \left\| g  \right\|_{L^p([x_0,b], \mu_1)}
& \le
c \left(  \left\| g  \right\|_{L^p([x_0,b], \mu_0)} +
 \left\|  g'  \right\|_{L^p([x_0,b], \mu_1)} \right) +
 \left| g(x_0)  \right|\left( \mu_1([x_0,b])^{1/p} + c\, \mu_0([x_0,b])^{1/p} \right) ,
\\
 \left\| g \right\|_{L^p([a,x_0], \mu_1)}
& \le
c \left(  \left\| g  \right\|_{L^p([a,x_0], \mu_0)} +
 \left\|  g'  \right\|_{L^p([a,x_0], \mu_1)} \right) +
 \left| g(x_0)  \right|\left( \mu_1([a,x_0])^{1/p} + c\, \mu_0([a,x_0])^{1/p} \right) ,
\end{aligned}
\end{equation}
hold for any  $g\in \mathbb{P}$.

Since $\mu_0\left( (a,b) \right) >0$, there exist $a<a_0<b_0<b$ with $x_0 \in [a_0,b_0]$ and $\mu_0\left( [a_0,b_0] \right)>0$.
Taking into account that $w_1^{-1/(p-1)}\in L^{1}([a_0,b_0])$, by
Lemma \ref{l:bp}
there exists a  constant $c_1>0$ such that
$$
\left|g(x_0)\right|\le
c_1 \, \left(\left\|g\right\|_{L^p([a_0,b_0],\mu_0)}+\left\|g'\right\|_{L^p([a_0,b_0],\mu_1)} \right),\quad \forall \, g\in\mathbb{P}.
$$

By this inequality and \eqref{eq:a6}, there exists a  constant $c_2>0$ such that
$$
\left\|  g\right\|_{L^p( \mu_1)}
\le
c_2 \left( \left\| g\right\|_{L^p( \mu_0)} +
\left\|  g' \right\|_{L^p( \mu_1)} \right),\quad\forall\, g\in\mathbb{P},
$$
as a consequence, the MO is bounded in $\PP^{1,p}(\mu_0, \mu_1)$ by \eqref{maininequality00}.
\end{proof}

\medskip

\begin{proof}[\it Proof of Theorem \ref{t:subinterval}]

\

$\bullet$ We prove first that if $\mu_0\left( [a,b] \right)=0$, then
the MO is not bounded in $\PP^{1,p}(\mu_0, \mu_1)$.
By contradiction, let us suppose that the MO is bounded in $\PP^{1,p}(\mu_0, \mu_1)$.
Therefore  by \eqref{maininequality00} there exists a constant  $c>0$ such that
$$
\left\|f\right\|_{L^p(\mu_1)} \le c \left( \left\|f\right\|_{L^{p}(\mu_0)} + \left\|f' \right\|_{L^{p}(\mu_1)}  \right), \quad\forall\, f\in\mathbb{P}.
$$
Taking $f=1$, we obtain
$$
\mu_1\left( [a,b] \right)
\le c^p \mu_0\left([a,b]\right)=0\, ,
$$
then we conclude $\mu_1\left( [a,b] \right)=0$, which is a contradiction with $w_1^{-1/(p-1)}\in L^1([a+\e,b-\e])$ for every $0<\e<(b-a)/2$,
and we deduce that the MO is not bounded in $\PP^{1,p}(\mu_0, \mu_1)$.


\smallskip

In order to prove part $(1)$ we assume that $\mu_0\left( [a,b] \right)>0$.
Since $w_1^{-1/(p-1)}\in L^{1}([a,b])$, by Lemma \ref{l:bp}
 there exists a  constant $c_2>0$ such that
$$
\left\|g\right\|_{L^p(\mu_1)}\le
c_2 \, \left(\left\|g\right\|_{L^p(\mu_0)}+\left\|g'\right\|_{L^p(\mu_1)} \right),\quad\forall\, g\in\mathbb{P}.
$$
Therefore the MO is bounded in $\PP^{1,p}(\mu_0, \mu_1)$.

\medskip

$\bullet$ In order to prove part $(2)$, we assume that
the MO is bounded in $\PP^{1,p}(\mu_0, \mu_1)$.
Then, there exists a constant $c>0$ such that
\begin{equation} \label{eq:a0}
\left\|  g\right\|_{L^p( \mu_1)}
\le
c \left( \left\| g\right\|_{L^p( \mu_0)} +
\left\|  g' \right\|_{L^p( \mu_1)} \right),\quad\forall\, g\in\mathbb{P}.
\end{equation}
In particular we have
\begin{equation} \label{eq:a1}
\left\|  \dint_{a}^x f(t) \, dt \right\|_{L^p( \mu_1)}
\le
c \left( \left\|  \dint_{a}^x f(t) \, dt \right\|_{L^p( \mu_0)} +
\left\|  f \right\|_{L^p( \mu_1)} \right),\quad\forall\, f\in\mathbb{P}.
\end{equation}
By Theorem \ref{l:polMuckenhoupt} we know that \eqref{eq:a1}
holds for all $f\in\cM ([a,b])$.
Hence, applying Corollary \ref{c:iff} we obtain
\eqref{eq:a02} for some constant $k$.

\smallskip

Let's assume now that $\mu_0\left( [a,b] \right)>0$ and that
\eqref{eq:a02} holds for some constant $k$.
By Proposition \ref{p:first} there exists a constant $c>0$ such that
$$
\left\|  \dint_a^{x} f(t) \, dt \right\|_{L^p( \mu_1)}
\le c \left( \left\|  \dint_a^{x} f(t) \, dt \right\|_{L^p( \mu_0)} +
\left\| f\right\|_{L^p( \mu_1)} \right),\quad \forall\, f\in\cM ([a,b]).
$$
Then
$$
\left\| g - g(a) \right\|_{L^p( \mu_1)}
\le
c \left( \left\| g - g(a) \right\|_{L^p( \mu_0)} +
\left\| g'\right\|_{L^p( \mu_1)} \right),\quad\forall\, g\in\mathbb{P}.
$$
Consequently,
\begin{equation} \label{eq:a5}
\left\| g\right\|_{L^p( \mu_1)}
\le
c \left( \left\| g \right\|_{L^p( \mu_0)} +
\left\|  g' \right\|_{L^p( \mu_1)} \right) +
\left| g(a) \right|\left( \mu_1([a,b])^{1/p} + c\, \mu_0([a,b])^{1/p} \right),\quad\forall\, g\in\mathbb{P}.
\end{equation}
Since $\mu_0\left( [a,b) \right)>0$, there exists $a<b_0<b$ with $\mu_0\left( [a,b_0] \right)>0$.
Taking into account that $w_1^{-1/(p-1)}\in L^{1}([a,b_0])$, by
Lemma \ref{l:bp} there exists a constant $c_1>0$ such that
$$
\left|g(a)\right|\le
c_1 \left(\left\|g\right\|_{L^p([a,b_0],\mu_0)}+\left\|g'\right\|_{L^p([a,b_0],\mu_1)} \right),\quad\forall\, g\in\mathbb{P}.
$$
This inequality jointly with \eqref{eq:a5} give \eqref{eq:a0}, and then
the MO is bounded in $\PP^{1,p}(\mu_0, \mu_1)$.

\medskip

A similar  argument  to the one in part $(2)$ allows us to prove part $(3)$.

\medskip

$\bullet$ Finally, let us prove part $(4)$.
Fix $x_0\in (a,b)$.
Assume first that the MO is bounded in $\PP^{1,p}(\mu_0, \mu_1)$.
Then \eqref{eq:a0} holds for every polynomial.
In particular,
\begin{equation} \label{eq:a30}
\left\|  \dint_{x_0}^x f(t) \, dt \right\|_{L^p(\mu_1)}
\le
c \left( \left\|  \dint_{x_0}^x f(t) \, dt \right\|_{L^p( \mu_0)} +
\left\| f\right\|_{L^p( \mu_1)} \right),\quad\forall\, f\in\mathbb{P}.
\end{equation}
We are going to prove that
\begin{equation} \label{eq:a3}
\begin{aligned}
\left\|  \dint_{x_0}^x f(t) \, dt \right\|_{L^p([x_0,b], \mu_1)}
& \le
c \left( \left\|  \dint_{x_0}^x f(t) \, dt \right\|_{L^p([x_0,b], \mu_0)} +
\left\|  f \right\|_{L^p([x_0,b], \mu_1)} \right)
\\
\left\|  \dint_{x_0}^x f(t) \, dt \right\|_{L^p([a,x_0], \mu_1)}
& \le
c \left( \left\|  \dint_{x_0}^x f(t) \, dt \right\|_{L^p([a,x_0], \mu_0)} +
\left\|  f \right\|_{L^p([a,x_0], \mu_1)} \right)
\end{aligned}
\end{equation}
hold for every $f\in\mathbb{P}$.
By symmetry, it suffices to prove the first inequality.
for $f\in \mathbb{P}$ and $\e>0$, by the density of the continuous functions in $L^p$, there exists a function $h_0\in \mathcal{C}([a,b])$
with
$$
\left\| fI_{[x_0,b]}-h_0 \right\|_{L^p(\mu_0+\mu_1)} + \left\| fI_{[x_0,b]}-h_0 \right\|_{L^1([a,b])} < \e\,.
$$
Weierstrass' Theorem provides a polynomial $h$ with $\| h_0-h \|_{L^\infty([a,b])} < \e$.
Let us define the constant $c_0:=(\mu_0+\mu_1)([a,b])^{1/p}$.
We have
$$
\begin{aligned}
\left\| h_0-h \right\|_{L^p(\mu_0+\mu_1)}
& + \left\| h_0-h \right\|_{L^1([a,b])}
< \e\, c_0 + \e\, (b-a)
\,,
\\
\left\| fI_{[x_0,b]}-h \right\|_{L^p(\mu_0+\mu_1)}
& + \left\| fI_{[x_0,b]}-h \right\|_{L^1([a,b])}
< (c_0+b-a+1)\, \e
\,,
\\
\left\|  \dint_{x_0}^x f(t) \, dt \, I_{[x_0,b]}(x)- \dint_{x_0}^x h(t) \, dt \right\|_{L^p(\mu_i)}
& = \left\|  \dint_{x_0}^x f(t)I_{[x_0,b]}(t) \, dt - \dint_{x_0}^x h(t) \, dt \right\|_{L^p(\mu_i)}
\\
& \le
\left\|  fI_{[x_0,b]}-h \right\|_{L^1([a,b])} \mu_i([a,b])^{1/p}
< c_0\,(c_0+b-a+1)\, \e\,.
\end{aligned}
$$
Since $\e>0$ is arbitrary, these inequalities and \eqref{eq:a30} prove the first one in \eqref{eq:a3}.

\smallskip

Moreover, taking into account that the inequalities \eqref{eq:a3} hold for all polynomials, then by
Theorem \ref{l:polMuckenhoupt}  both inequalities in  \eqref{eq:a3}
hold for all  $f\in \cM ([a,b])$.
Hence, by Corollary \ref{c:iff} we obtain
\eqref{eq:a04} for some constant $k$.

\smallskip

If we assume that \eqref{eq:a04} holds, then by Lemma \ref{l:last} we conclude.
\end{proof}

\medskip

Now, we consider a new kind of measures.

\medskip

\begin{definition} \label{d:piecewiseregular}
Let  $\mu_1$ be a measure on $\RR$.
We say that $\mu_1$ is \emph{piecewise regular} if there exist real numbers $a_0<a_1<\cdots<a_m$
verifying the following properties:

$(a)$ The convex hull of the support of $\mu_1$ is the compact interval $[a_0,a_m]$.

$(b)$ If   $1\le j\le m$ we have either
$w_1^{-1/(p-1)}\in L^{1}([a_{j-1}+\e,a_{j}-\e])$ for every $0<\e<(a_{j}-a_{j-1})/2$
or $\mu_1|_{(a_{j-1},a_{j})}$ is a finite linear combination of Dirac deltas.

$(c)$ For  $0< j< m$ we have $w_1^{-1/(p-1)}\notin L^{1}([a_{j}-\e,a_{j}+\e])$ for every $\e>0$
and we do not have $w_1=0$ a.e. in $[a_{j-1},a_{j+1}]$.

\smallskip

We say that $\mu_1$ is \emph{strongly piecewise regular}
if it is piecewise regular and it verifies the following property: for  $0\le j\le m$,
if $w_1^{-1/(p-1)}\notin L^{1}([a_{j}-\e,a_{j}])$ then $w_1^{-1/p}\notin L^{1}([a_{j}-\e,a_{j}])$, and
if $w_1^{-1/(p-1)}\notin L^{1}([a_{j},a_{j}+\e])$ then $w_1^{-1/p}\notin L^{1}([a_{j},a_{j}+\e])$.

\smallskip

We define $J$ as the set of indices $1\le j\le m$ with $w_1^{-1/(p-1)}\in L^{1}([a_{j-1}+\e,a_{j}-\e])$ for every $0<\e<(a_{j}-a_{j-1})/2$,
while $H$ will be the (finite) set of points $x\in [a_0,a_m]$ verifying $\mu_1 (\{x\}) > 0$,
$w_1^{-1/(p-1)} \notin L^{1}([x-\e,x])$ and $w_1^{-1/(p-1)}\notin L^{1}([x,x+\e])$
for every $\e> 0$.
\end{definition}

\medskip

\begin{Remarks}

\

\begin{enumerate}
\item Condition $(c)$
is not a real restriction, since it just guarantees
the uniqueness of $m$ and the real numbers $a_0<a_1<\cdots<a_m$.

\item
By condition $(b)$ we have that either $w_1=0$ a.e. in $[a_{j-1},a_{j}]$ or $w_1 > 0$ a.e. in $[a_{j-1},a_{j}]$, for $1\le j\le m$.

\item
If $x\in \cup_{j\notin J} (a_{j-1},a_{j})$ verifies $\mu_1 (\{x\}) > 0$, then $x\in H$.

\item Note also that
the practical totality of the measures with compact support in $\RR$ which one usually deals with
in the study of orthogonal polynomials is strongly piecewise regular
(see Example \ref{exam1}).

\item The class of piecewise regular measures allows us to consider (and this was not the case in the papers
\cite{RARP2}, \cite{R4} and \cite{RS}) measures for which the Sobolev space $W^{1,p}([a,b],(\mu_0, \mu_1))$
is not well defined and $\PP^{1,p}([a,b],(\mu_0, \mu_1))$ is not a space of functions
(e.g.,
$\|f\|_{W^{1,p}([-1,1],\mu_0,\mu_1)}^p := \int_{-1}^1 |f(x)|^p dx + \int_{-1}^1 |f'(x)|^p (1-x^2)^{p-1} dx + \a|f'(-1)|^p + \b|f'(1)|^p$
with $\a,\b\ge 0$ and $\a+\b>0$, or the example at the introduction
$\|f\|_{W^{1,p}([0,1],\mu_0,\mu_1)}^p := \int_0^1 |f(x)|^p dx + |f(0)|^p + |f'(0)|^p$).
\end{enumerate}
\end{Remarks}

\medskip

The following result (see \cite[Theorem 4.4]{R2}) will be necessary:

\medskip

\begin{theorem} \label{t:R1}
Let  $\mu_0, \mu_1$ be finite measures on $[a,b]$, and $\a\in [a,b]$.
Assume that $w_1^{-1/(p-1)}\notin L^{1}([\a-\e,\a])$ and
$w_1^{-1/(p-1)}\notin L^{1}([\a,\a+\e])$ for every $\e>0$.
If $\mu_1(\{\a\})>0$ and $\mu_0(\{\a\})=0$, then the
MO is not bounded in $\PP^{1,p}([a,b],(\mu_0, \mu_1))$.
\end{theorem}

\medskip

\begin{theorem} \label{t:subinterval2}
Let $\mu_0, \mu_1$ be finite measures on $[a,b]$.

\smallskip

$(1)$
Assume that $\mu_1$ is piecewise regular with $a_0<a_1<\cdots<a_m$.
If the MO is bounded in  the space $\PP^{1,p}(\mu_0|_{Reg([a_{j-1},a_{j}])}, \mu_1|_{Reg([a_{j-1},a_{j}])})$
for each $j\in J$ and $\mu_0(\{x\})>0$ for all $x\in H$,
then it is bounded in $\PP^{1,p}(\mu_0, \mu_1)$.

\smallskip

$(2)$
Suppose that $\mu_1$ is strongly piecewise regular with $a_0<a_1<\cdots<a_m$.
If the MO is bounded in $\PP^{1,p}(\mu_0, \mu_1)$,
then it is bounded in $\PP^{1,p}(\mu_0|_{Reg([a_{j-1},a_{j}])}, \mu_1|_{Reg([a_{j-1},a_{j}])})$
for each $j\in J$ and $\mu_0(\{x\})>0$ for all $x\in H$.
\end{theorem}

\medskip

\begin{proof}
First of all, note that by the definition of piecewise regular measure we have:
\begin{equation} \label{eq:a8}
\mu_1 = \sum_{j\in J} \mu_1|_{Reg([a_{j-1},a_{j}])} + \sum_{x\in H} \mu_1(\{x\}) \d_{x} \,.
\end{equation}

\smallskip

Let's assume that $\mu_1$ is piecewise regular, that the MO is bounded in $\PP^{1,p}(\mu_0|_{Reg([a_{j-1},a_{j}])}, \mu_1|_{Reg([a_{j-1},a_{j}])})$
for each $j\in J$ and that $\mu_0(\{x\})>0$ for every $x\in H$.
Then, for every $j\in J$, there exists a constant $c_j$ such that (see \eqref{maininequality00})
$$
\begin{aligned}
 \left\| g  \right\|_{L^p(\mu_1|_{Reg([a_{j-1},a_{j}])})}
& \le
c_j \left(  \left\| g   \right\|_{L^p(\mu_0|_{Reg([a_{j-1},a_{j}])})} +
 \left\|  g'  \right\|_{L^p( \mu_1|_{Reg([a_{j-1},a_{j}])})} \right)
\\
& \le
c_j \left(  \left\| g   \right\|_{L^p( \mu_0)} +
 \left\|  g'  \right\|_{L^p( \mu_1)} \right) ,
\\
 \left\| g \right\|_{L^p( \mu_1|_{H})}
& \le
\left(\frac{\max_{x\in H} \mu_1(\{x\})}{\min_{x\in H} \mu_0(\{x\})}\right)^{1/p}  \left\| g  \right\|_{L^p( \mu_0|_{H})}
\\
& \le
\left(\frac{\max_{x\in H} \mu_1(\{x\})}{\min_{x\in H} \mu_0(\{x\})}\right)^{1/p} \left\| g\right\|_{L^p( \mu_0)} \,,
\end{aligned}
$$
are true for every $g\in\mathbb{P}$.

These inequalities and \eqref{eq:a8} give \eqref{eq:a0}, and then the MO is bounded in $\PP^{1,p}(\mu_0, \mu_1)$.

\medskip

Let's assume now that $\mu_1$ is strongly piecewise regular and that the MO is bounded in $\PP^{1,p}(\mu_0, \mu_1)$.
Therefore, there exists a positive constant $c$ such that
\begin{equation} \label{eq:prea8}
\left\|f\right\|_{L^p(\mu_1)} \le c \left( \left\|f\right\|_{L^{p}(\mu_0)} + \left\|f' \right\|_{L^{p}(\mu_1)}  \right),\quad\forall\, f\in \mathbb{P}.
\end{equation}
Since the MO is bounded in $\PP^{1,p}(\mu_0, \mu_1)$, \eqref{eq:a8} and Theorem \ref{t:R1} prove that
$\mu_0(\{x\})>0$ for every $x\in H$.

\medskip

Let us prove now that the MO is bounded in
$\PP^{1,p}(\mu_0|_{Reg([a_{j-1},a_{j}])}, \mu_1|_{Reg([a_{j-1},a_{j}])})$
for each $j\in J$.
Let us fix $j\in J$.
Note that in order to check that there exists a constant $c>0$ such that
$$
 \left\|g \right\|_{L^p(\mu_1|_{Reg([a_{j-1},a_{j}])})} \le c \left( \left\|g \right\|_{L^{p}(\mu_0|_{Reg([a_{j-1},a_{j}])})} +  \left\|g' \right\|_{L^{p}(\mu_1|_{Reg([a_{j-1},a_{j}])})}  \right),\quad\forall\, g\in \mathbb{P},
$$
 by \eqref{eq:prea8} it suffices to prove that given $g\in \mathbb{P}$ and $\e>0$,
there exists  $g_0\in \mathbb{P}$ with
\begin{equation} \label{eq:a9}
\left| \left\|g\right\|_{L^p(\mu_i|_{Reg([a_{j-1},a_{j}])})} -  \left\|g_0  \right\|_{L^p(\mu_i)} \right| \le \e \,,
\quad
\left| \left\|g' \right\|_{L^p(\mu_1|_{Reg([a_{j-1},a_{j}])})} - \left\|g_0' \right\|_{L^p(\mu_1)} \right| \le \e \,.
\end{equation}

Assume first that $Reg([a_{j-1},a_{j}])=[a_{j-1},a_{j}]$.
Then we have $w_1^{-1/p}\notin L^{1}([a_{j-1}-\e',a_{j-1}])$ and
$w_1^{-1/p}\notin L^{1}([a_{j},a_{j}+\e'])$ for every $\e'>0$.
Fixed  $g\in \mathbb{P}$ and $\e>0$, we define $K:=\max\{|g(a_{j-1})|,\,|g(a_{j})|\}$.
Since $\mu_1, \mu_2$ are finite, there exists $0<\eta< \e^p$ with
$$
\mu_0\left((a_{j-1}-\eta,a_{j-1})\right), \; \; \mu_1\left((a_{j-1}-\eta,a_{j-1})\right), \; \; \mu_0\left((a_{j},a_{j}+\eta)\right), \; \; \mu_1\left((a_{j},a_{j}+\eta)\right) < \e^p .
$$
Since $w_1^{-1/p}\notin L^{1}\left([a_{j-1}-\eta,a_{j-1}]\right)$ and
$w_1^{-1/p}\notin L^{1}([a_{j},a_{j}+\eta])$, there exist $t_1,t_2 > 0$ verifying
$$
\int_{a_{j-1}-\eta}^{a_{j-1}} \min \left\{w_1^{-1/p},\, t_1 \right\} = 1 \,,
\qquad
\int_{a_{j}}^{a_{j}+\eta} \min \left\{w_1^{-1/p},\, t_2 \right\} = 1 \,.
$$
Let's define $g_1:=g(a_{j-1})\min \left\{w_1^{-1/p},\, t_1 \right\}$
and $g_2:= g(a_{j})\min \left\{w_1^{-1/p},\, t_2 \right\}$.
Fixed a measurable set $S\subseteq [a,b]$ with
zero Lebesgue measure and such that $(\mu_1)_s|_S=(\mu_1)_s$.
Let $f_1(x):=\int_a^x f_0$, where $f_0$ is defined by the following
$$
f_0:=
\left\{
\begin{aligned}
0 \,, \quad & \text{ on } (-\infty,a_{j-1}-\eta] \,,
\\
g_1 I_{\RR\setminus S}\,, \quad & \text{ on } (a_{j-1}-\eta,a_{j-1}) \,,
\\
g' \,, \quad & \text{ on } [a_{j-1},a_{j}] \,,
\\
-g_2 I_{\RR\setminus S}\,, \quad & \text{ on } (a_{j},a_{j}+\eta) \,,
\\
0 \,, \quad & \text{ on } [a_{j}+\eta,\infty) \,.
\end{aligned}
\right.
$$
We have $f_1=g$ on $[a_{j-1},a_{j}]$ and $f_1=0$ on
$(-\infty,a_{j-1}-\eta] \cup [a_{j}+\eta,\infty)$.
Hence,
$$
\begin{aligned}
\left| \left\|g\right\|_{L^p(\mu_i|_{Reg([a_{j-1},a_{j}])})} -  \left\|f_1 \right\|_{L^p(\mu_i)} \right|
&
\le  \left\| f_1  \right\|_{L^p(\mu_i|_{(a_{j-1}-\eta,a_{j-1})})}
+  \left\| f_1  \right\|_{L^p(\mu_i|_{(a_{j},a_{j}+\eta)})}
\\
& \le K \, \mu_i\left((a_{j-1}-\eta,a_{j-1})\right)^{1/p}  + K \, \mu_i\left((a_{j},a_{j}+\eta)\right)^{1/p}
\le 2 \, K \, \e
\,,
\\
\left| \left\|g'\right\|_{L^p(\mu_1|_{Reg([a_{j-1},a_{j}])})} - \left\|f_0\right\|_{L^p(\mu_1)} \right| &
\le \left\| g_1 \right\|_{L^p(w_1|_{(a_{j-1}-\eta,a_{j-1})})}
+ \left\| g_2\right\|_{L^p(w_1|_{(a_{j},a_{j}+\eta)})}
\\
& \le K \, \left(\int_{a_{j-1}-\eta}^{a_{j-1}} \left|w_1^{-1/p}\right|^p w_1 \right)^{1/p}
+ K \, \left(\int_{a_{j}}^{a_{j}+\eta} \left|w_1^{-1/p}\right|^p w_1 \right)^{1/p}
\\
& = 2 \, K \, \eta^{1/p}
\le 2 \, K \, \e
\,.
\end{aligned}
$$
Since $f_0\in L^1([a,b])$, there exists $h_0\in \mathcal{C}([a,b])$ with
$$
 \left\| f_0-h_0 \right\|_{L^p(\mu_0+\mu_1)} +  \left\| f_0-h_0 \right\|_{L^1([a,b])} < \e\,.
$$
Weierstrass's Theorem provides a polynomial $h$ with $ \| h_0-h \|_{L^\infty([a,b])} < \e$.
Let us define the constant $c_0:=(\mu_0+\mu_1)([a,b])^{1/p}$ and the polynomial $f_2(x):=\int_a^x h$.
Then,
$$
\begin{aligned}
 \left\| h_0-h  \right\|_{L^p(\mu_0+\mu_1)}
& +  \left\| h_0-h  \right\|_{L^1([a,b])}
< \e\, c_0 + \e\, (b-a)
\,,
\\
 \left\| f_0-h  \right\|_{L^p(\mu_0+\mu_1)}
& + \left\| f_0-h  \right\|_{L^1([a,b])}
< (c_0+b-a+1)\, \e
\,,
\\
\left\|  f_1- f_2  \right\|_{L^p(\mu_i)}
& \le
 \left\|  f_0-h  \right\|_{L^1([a,b])} \mu_i([a,b])^{1/p}
< c_0\,(c_0+b-a+1)\, \e\,,
\end{aligned}
$$
and we conclude
$$
\begin{aligned}
\left|  \left\|g \right\|_{L^p(\mu_i|_{Reg([a_{j-1},a_{j}])})} - \left\|f_2 \right\|_{L^p(\mu_i)} \right|
& \le (2 \, K + c_0\,(c_0+b-a+1) )\, \e
\,,
\\
\left| \left\|g'\right\|_{L^p(\mu_1|_{Reg([a_{j-1},a_{j}])})} - \left\|f_2'\right\|_{L^p(\mu_1)} \right|
& \le (2 \, K + c_0+b-a+1 ) \, \e
\,.
\end{aligned}
$$
Since $\e>0$ is arbitrary, this proves \eqref{eq:a9} when $Reg([a_{j-1},a_{j}])=[a_{j-1},a_{j}]$.

\medskip

If $Reg([a_{j-1},a_{j}])=(a_{j-1},a_{j})$,  we can apply the same argument considering now
the functions $g_1$ and $g_2$ in the intervals $(a_{j-1},a_{j-1}+\eta)$ and $(a_{j}-\eta,a_{j})$,
respectively.

\smallskip

In the case $Reg([a_{j-1},a_{j}])=[a_{j-1},a_{j})$, we consider  $g_1$ and $g_2$ in  $(a_{j-1}-\eta,a_{j-1})$ and $(a_{j}-\eta,a_{j})$,
respectively.

\smallskip

Finally, for  $Reg([a_{j-1},a_{j}])=(a_{j-1},a_{j}]$, we take $g_1$ and $g_2$ in $(a_{j-1},a_{j-1}+\eta)$ and $(a_{j},a_{j}+\eta)$,
respectively.
\end{proof}
Theorem \ref{t:subinterval2} has the following consequence.
\begin{corollary} \label{t:subinterval3}
Let $\mu_0, \mu_1$ be finite measures on $[a,b]$
such that $\mu_1$ is strongly piecewise regular with $a_0<a_1<\cdots<a_m$. Then the MO 
is bounded in $\PP^{1,p}(\mu_0, \mu_1)$
iff it is bounded in $\PP^{1,p}(\mu_0|_{Reg([a_{j-1},a_{j}])},\mu_1|_{Reg([a_{j-1},a_{j}])})$
for each $j\in J$ and $\mu_0(\{x\})>0$ for all $x\in H$.
\end{corollary}

\medskip

As we mentioned in the introduction, \ref{t:subinterval} and Corollary \ref{t:subinterval3}
together characterize the boundedness of the MO
for a large class of measures which includes the most usual examples in the literature
of orthogonal polynomials.
It is remarkable that we require the hypothesis of strongly piecewise regular just for $\mu_1$.

\medskip

The following example shows a large class of measures verifying the hypotheses in Corollary \ref{t:subinterval3}.

\medskip

\begin{example} \label{exam1}
The measure $\mu_1$ below is finite and strongly piecewise regular
$$
d\mu_1:= |x-a_0|^{\a_0} |x-a_1|^{\a_1} \cdots |x-a_m|^{\a_m} v(x) I_{[a_0,a_m]}(x) \, dx + \sum_{j=1}^{r} c_j d\d_{x_j}\,,
$$
if $c_1,\dots ,c_r \ge 0,$ $x_1,\dots ,x_r \in [a_0,a_m]$,
$\a_0,\a_1,\dots ,\a_m>-1,$ $\a_0,\a_1,\dots ,\a_m\notin [p-1,p)$,
and there exists a constant $C\ge 1$ with $C^{-1} \le v(x) \le C$ for $x \in [a_0,a_m]$.
\end{example}

\medskip

If we study a particular (although very large) class of measures, it is possible to improve the first conclusion
in Theorem \ref{t:subinterval2}.

\medskip

\begin{definition}
Let $\mu_1$ be a measure on $\RR$.
We say that $\mu_1$ is \emph{piecewise monotone} if there exist real numbers $b_0<b_1<\cdots<b_n$
verifying the following properties:

$(a)$ The convex hull of the support of $\mu_1$ is the compact interval $[b_0,b_n]$.

$(b)$ For each $1\le j\le n$ the weight $w_1$ is comparable to a (non-strictly) monotone function
on $(b_{j-1},b_{j})$.

$(c)$ The singular part of $\mu_1$ is a finite linear combination of Dirac deltas.

If $\mu_1$ is piecewise monotone, then it is piecewise regular with constants $a_0<a_1<\cdots<a_m$ (see Definition \ref{d:piecewiseregular}).
We say that $a_0<a_1<\cdots<a_m$ are the \emph{parameters} of $\mu_1$.
\end{definition}

Note that if $\mu_1$ is piecewise monotone, then $b_0=a_0$ and $b_n=a_m$,
but it is possible that $\{b_0,b_1,\dots,b_n\}\neq \{a_0,a_1,\dots,a_m\}$. Also, it is possible that $w_1=0$ in some $(b_{j-1},b_{j})$.

\medskip

The following results are specially useful in the study of SOP,
as Example \ref{exam2} below shows.
\begin{theorem} \label{t:subinterval4}
Let  $\mu_0, \mu_1$ be finite measures on $[a,b]$,
where $\mu_1$ is piecewise monotone with parameters $a_0<a_1<\cdots<a_m$.
If $\mu_0\left( Reg([a_{j-1},a_{j}]) \right)>0$
for each $j\in J$ and $\mu_0(\{x\})>0$ for all $x\in H$,
then the MO is bounded in $\PP^{1,p}(\mu_0, \mu_1)$.
\end{theorem}
\begin{proof}
By Theorem \ref{t:subinterval2}, it suffices to prove for every $j\in J$ that
the MO is bounded in  the space\\ $\PP^{1,p}(\mu_0|_{Reg([a_{j-1},a_{j}])}, \mu_1|_{Reg([a_{j-1},a_{j}])})$.

Fixed $j\in J$, since $\mu_1$ is piecewise monotone, there exists $0< \e < (a_{j}-a_{j-1})/2$ such that
$w_1$ is comparable to a (non-strictly) monotone function on $(a_{j-1},a_{j-1}+\e)$ and on $(a_{j}-\e,a_{j})$,
and $(\mu_1)_s \left( (a_{j-1},a_{j-1}+\e) \right)= (\mu_1)_s \left( (a_{j}-\e,a_{j}) \right)= 0$.

Assume that $Reg([a_{j-1},a_{j}])= (a_{j-1},a_{j})$, since the other cases are similar and easier.
Using that
$\mu_0\left( Reg([a_{j-1},a_{j}]) \right)>0$, by Lemma \ref{l:last}
the MO is bounded in $\PP^{1,p}(\mu_0|_{Reg([a_{j-1},a_{j}])}, \mu_1|_{Reg([a_{j-1},a_{j}])})$
if we have
$$
\Lambda_{p,[a_{j-1},x_0],a_{j-1}} \left( (\mu_1-k_0\mu_0)_+,\,\mu_1 \right)<\infty\,,
\quad
\Lambda_{p,[x_0,a_{j}],a_{j}} \left( (\mu_1-k_0\mu_0)_+,\,\mu_1 \right)
< \infty\,,
$$
for some constant $k_0$ and some point $x_0 \in (a_{j-1},a_{j})$.
By Lemma \ref{General Lemma} these inequalities are equivalent to
$$
\Lambda_{p,[a_{j-1},a_{j-1}+\e],a_{j-1}} \left( (\mu_1-k\mu_0)_+,\,\mu_1 \right)<\infty\,,
\quad
\Lambda_{p,[a_{j}-\e,a_{j}],a_j} \left( (\mu_1-k\mu_0)_+,\,\mu_1 \right)
< \infty\,,
$$
for some constant $k$.
Applying Theorem \ref{t:iff3} we obtain these inequalities, so the proof is finished.
\end{proof}

\medskip

The following example shows the large class of measures verifying the hypotheses in Theorem \ref{t:subinterval4}.
\begin{example} \label{exam2}
Given $a\in \RR$, let us consider the set $\mathfrak{W}_a$ of weights
obtained by the products of:
$$
|x-a|^{\a_1}\,, \quad \exp \left(-\beta|x-a|^{-\a_2}\right)\,,
\quad \left|\log \dfrac 1{|x-a|} \right|^{\a_3}\,, \quad
\left|\log \left|\log \left| \cdots \left|\log \dfrac 1{|x-a|}\right| \cdots \right| \right| \right|^{\a_4}\,,
$$
in such a way that the weights are integrable in some neighborhood of $a$,
and denote by $\mathfrak{W}$ the class of weights $w$ for which there exist $a_0<a_1<\cdots <a_m$ and weights $v_j \in \mathfrak{W}_{a_j}$
such that $w$ is comparable to $v_j$ in some neighborhood $V_j$ of $a_j$ for $j=0,1,\dots,m,$
and $w$ is comparable to the constant function $1$ in $[a_0,a_m]\setminus \cup_{j=0}^m V_j$.
We say that $a_0<a_1<\cdots<a_m$ are the parameters of $w$.
If
$$
d\mu_1:= w(x) I_{[a_0,a_m]}(x) \, dx + \sum_{j=1}^{r} c_j d\d_{x_j}\,,
$$
where $w \in \mathfrak{W}$ with parameters $a_0<a_1<\cdots<a_m$,
$c_1,\dots ,c_r \ge 0,$ and $x_1,\dots ,x_r \in [a_0,a_m]$,
then $\mu_1$ is finite and piecewise monotone.

Note that this class of measures is wider than the one in  Example \ref{exam1}.
\end{example}
As a consequence of Theorem \ref{t:subinterval4} we have the following result.
\begin{corollary} \label{t:subinterval5}
Let  $\mu_0, \mu_1$ be finite measures on $[a,b]$,
where $\mu_1=\mu_{1,1}+\mu_{1,2}$, $\mu_{1,1}$ is piecewise monotone with parameters $a_0<a_1<\cdots<a_m$
and $\mu_{1,2} \le k\mu_0$ for some constant $k$.
If $\mu_0\left( Reg([a_{j-1},a_{j}]) \right)>0$
for each $j\in J$ and $\mu_0(\{x\})>0$ for all $x\in H$,
then the MO is bounded in $\PP^{1,p}(\mu_0, \mu_1)$.
\end{corollary}

\

\end{document}